\documentclass[a4paper,10pt,reqno]{amsart}
\usepackage{amsmath,amsfonts,amscd,amssymb,graphicx,mathrsfs,eufrak,mathrsfs}

\usepackage{cite}
\usepackage[dvips,all,arc,curve,color,frame]{xy}
\usepackage{tikz}
\usetikzlibrary{arrows,decorations.pathmorphing,decorations.pathreplacing,positioning,shapes.geometric,shapes.misc,decorations.markings,decorations.fractals,calc,patterns}
\usepackage[colorlinks]{hyperref}
\usepackage{comment}

\tikzset{>=stealth',
        cvertex/.style={circle,draw=black,inner sep=1pt,outer sep=3pt},
        vertex/.style={circle,fill=black,inner sep=1pt,outer sep=3pt},
        star/.style={circle,fill=yellow,inner sep=0.75pt,outer sep=0.75pt},
        tvertex/.style={inner sep=1pt,font=\scriptsize},
        gap/.style={inner sep=0.5pt,fill=white}}

\usepackage{mathtools}

\author{J Karmazyn}
\title{Superpotentials, Calabi Yau algebras, and PBW deformations}

\theoremstyle{definition} 
\newtheorem{quiver}{Definitions}[subsection]
\newtheorem{pathalg}[quiver]{Definition}
\newtheorem{superpotential}[quiver]{Definition}
\newtheorem{superpotential2}[quiver]{Definition}
\newtheorem{diff}[quiver]{Definition}
\newtheorem{diff2}[quiver]{Definition}
\newtheorem{CY}[quiver]{Definition}

\newtheorem{dual2}[quiver]{Definition}

\newtheorem{NKoszul}[quiver]{Definition}
\newtheorem{PBW}[quiver]{Definition}
\newtheorem{coh}[quiver]{Definition}
\newtheorem{superp}[quiver]{Definition}

\newtheorem{example2}[quiver]{Example}

\newtheorem{example4}[quiver]{Example}
\newtheorem{example5}[quiver]{Example}

\newtheorem{examplePBW}[quiver]{Example}

\newtheorem{examplepp}[quiver]{Example}

\newtheorem{gl1}[quiver]{Example}
\newtheorem{gl2}[quiver]{Example}

\theoremstyle{plain} 
 
\newtheorem{Calabi2}[quiver]{Lemma}
\newtheorem{BSW1}[quiver]{Theorem}
\newtheorem{BSW2}[quiver]{Theorem} 
\newtheorem{BSWK}[quiver]{Theorem}
\newtheorem{BSW3}[quiver]{Theorem} 
\newtheorem{BT1}[quiver]{Theorem} 
\newtheorem{WZ}[quiver]{Theorem} 
\newtheorem{zeroPBW}[quiver]{Lemma} 
 
\newtheorem{Cor1}[quiver]{Remark}
\newtheorem{Extra}[quiver]{Theorem}

\newtheorem*{pTheorem3}{Theorem} 
\newtheorem{Theorem3}[quiver]{Theorem} 
\newtheorem*{pTheoremCY}{Theorem} 
\newtheorem{TheoremCY}[quiver]{Theorem} 
\newtheorem{TheoremBT}[quiver]{Theorem} 
\newtheorem{EG}[quiver]{Theorem} 
\newtheorem{Sym}[quiver]{Theorem} 
\newtheorem*{pSym}{Theorem}

\newtheorem{glcase}[quiver]{Theorem}
\newtheorem{Gen}[quiver]{Theorem}

\newtheorem{Class}{Lemma}[subsection]

\DeclareMathAlphabet{\mathbbm}{U}{bbm}{m}{n}

\newcommand{\K}{\mathbbm{k}}
\newcommand{\A}{\mathcal{A}}

\let\oldtocsection=\tocsection
\let\oldtocsubsection=\tocsubsection
\let\oldtocsubsubsection=\tocsubsubsection
\renewcommand{\tocsection}[2]{\hspace{0em}\oldtocsection{#1}{#2}}
\renewcommand{\tocsubsection}[2]{\hspace{1em}\oldtocsubsection{#1}{#2}}
\renewcommand{\tocsubsubsection}[2]{\hspace{2em}\oldtocsubsubsection{#1}{#2}}

\newcommand{\Hom}{\textrm{Hom}}
\newcommand{\RHom}{\textrm{RHom}}

\newcommand{\SL}{\textrm{SL}}
\newcommand{\GL}{\textrm{GL}}

\def\Im{\mathop{\rm Im}\nolimits}

\begin{document}

\maketitle

\begin{abstract} The paper \cite{BSW} by Bocklandt, Schedler and Wemyss considers path algebras with relations given by the higher derivations of a superpotential, giving a condition for such an algebra to be Calabi-Yau. In particular they show that  the algebra $\mathbb{C}[V]\rtimes G$, for $V$ a finite dimensional $\mathbb{C}$ vector space and $G$ a finite subgroup of $\GL(V)$, is Morita equivalent to a path algebra with relations given by a superpotential, and is Calabi-Yau for $G<\SL(V)$. In this paper we extend these results, giving a condition for a PBW deformation of a Calabi-Yau, Koszul path algebra with relations given by a superpotential to have relations given by a superpotential, and proving these are Calabi-Yau in certain cases.

We apply our methods to symplectic reflection algebras, where we show that every symplectic reflection algebra is Morita equivalent to a path algebra whose relations are given by the higher derivations of an inhomogeneous superpotential.  In particular we show these are Calabi-Yau regardless of the deformation parameter.

Also, for $G$ a finite subgroup of $\GL_2(\mathbb{C})$ not contained in $\SL_2(\mathbb{C})$, we consider PBW deformations of a path algebra with relations which is Morita equivalent to $\mathbb{C}[x,y] \rtimes G$. We show there are no non trivial PBW deformations when $G$ is a small subgroup.

\end{abstract}

\tableofcontents

\section{Introduction}

\subsection{Introduction}

In this paper we consider path algebras of quivers with certain relations, in particular studying relations produced from a superpotential.  Given a quiver $Q$, a homogeneous superpotential of degree $n$ is an element, $\Phi_n=\sum c_{a_1 \dots a_n} a_1 \dots a_n$, in the path algebra of $Q$ satisfying the \emph{$n$ superpotential condition}: $c_{aq}=(-1)^{n-1}c_{qa}$ for all arrows $a$ and paths $q$. From such a superpotential $\Phi_n$ and a non-negative integer $k$ we construct an algebra $\mathcal{D}(\Phi_n,k):=\frac{\mathbb{C}Q}{R}$ as a path algebra with relations $R$. These relations are constructed by the process of \emph{differentiation}, where we define the left derivative of a path $p$ by a path $q$, denoted $\delta_q p$, to be $t$ if $p=qt$ and $0$ otherwise, and the relations  are given by $R=< \{ \delta_p \Phi_n : |p|=k\} >$.

Algebras of this form are considered in \cite{BSW}, where they are related to Calabi-Yau (CY), $N$-Koszul algebras. In \cite{BSW} a complex, $\mathcal{W}^{\bullet}$, is defined which depends only on the superpotential, and a path algebra with relations is $N$-Koszul and CY if and only if it is of the form  $\mathcal{D}(\Phi_n,k)$ for a superpotential $\Phi_n$ and $\mathcal{W}^{\bullet}$ is a resolution.

Skew group algebras, $\mathbb{C}[V] \rtimes G$, for $G$ a finite subgroup of $\GL(V)$, are Morita equivalent to path algebras of this form. These are $2$-Koszul, and CY when $G < \SL(V)$, and hence their relations can be given by a superpotential. An explicit way to calculate this superpotential is given in \cite[Theorem 3.2]{BSW}.

We prove two results concerning the PBW deformations of $(n-k)$-Koszul, $(k+2)$-CY algebras of the form $\mathcal{D}(\Phi_n,k)$. We define an inhomogeneous superpotential of degree $n$ to be an element of the path algebra $\Phi':=\Phi_n + \phi_{n-1} + \dots + \phi_k$, such that each $\phi_j:=\sum c_p p$ is a sum of elements of the path algebra of length $j$, and each $\phi_j$ satisfies the $n$ superpotential condition. Such a superpotential defines relations $P=<\{ \delta_p \Phi' : |p|=k \}>$, and we define $\mathcal{D}(\Phi',k):=\frac{\mathbb{C}Q}{P}$.  Theorem \ref{Theorem3} classifies which PBW deformations are of this form. This is known for the case of PBW deformations of $\mathcal{D}(\Phi_n,1)$, \cite[Theorem 3.1, 3.2]{BT}, but here we extend this to include higher differentials, $k>1$.

We then prove that certain classes of PBW deformations of a $2$-Koszul, $n$-CY ${\mathcal{D}(\Phi_n,n-2)}$ are $n$-CY in Theorem \ref{TheoremCY}. This is already known in the $3$-CY case due to\cite{BT} which proves that any PBW deformation of an $N$-Koszul, 3-CY $\mathcal{D}(\Phi_n,1)$ given by inhomogenous superpotential is 3-CY,\cite[Theorem 3.6]{BT}, and in the case of a one vertex quiver there is a result, \cite[Theorem 3.1]{WZ}, which finds a necessary and sufficient condition for a PBW deformation of a Noetherian, $2$-Koszul, $n$-CY algebra to be $n$-CY.

Next we give an application of these results to symplectic reflection algebras. Symplectic reflection algebras are defined in \cite{EG} as PBW deformations of  certain skew group algebras $\mathbb{C}[V] \rtimes G$, and  hence we consider the Morita equivalent path algebra with relations. Applying the previous results, and a result of \cite{EG}, we deduce that these path algebras with relations are of the form $\mathcal{D}(\Phi_{2n}+\phi_{2n-2},2n-2)$ and are $2n$-CY.

We go on to consider PBW deformations of the path algebras with relations Morita equivalent to $\mathbb{C}[\mathbb{C}^2] \rtimes G$ when $G$ is a finite subgroup of $\GL_2(\mathbb{C})$ not contained is $\SL_2(\mathbb{C})$. We show that if $G$ does not contain pseudo-reflections there are no PBW deformations.

\subsection{Main results}

1) A classification of the PBW deformations of a $(k+2)$-CY, $(n-k)$-Koszul, superpotential algebra $\mathcal{D}(\Phi_n,k)$, whose relations are given by inhomogeneous superpotentials. These are proved to be the PBW deformations satisfying one additional property, which we call the zeroPBW condition, Definition \ref{superPBW}, and the corresponding superpotentials are shown to be $k$-coherent superpotentials, Definition \ref{coherent}.

\begin{pTheorem3}(Theorem \ref{Theorem3})\label{pTheorem3}
Let $A=\mathcal{D}(\Phi_n,k)$, for $\Phi_n$ a homogeneous superpotential of degree n, be $(n-k)$-Koszul and $(k+2)$-CY. Then the zeroPBW deformations of $A$ correspond exactly to the algebras $\mathcal{D}(\Phi',k)$ defined by $k$-coherent inhomogeneous superpotentials of the form $\Phi'=\Phi_n +\phi_{n-1} + \dots + \phi_k$.
\end{pTheorem3}

2) A proof that these PBW deformations are CY in certain cases. We consider $A=\mathcal{D}(\Phi_n,n-2)$ which is $n$-CY and $2$-Koszul, and a zeroPBW deformation, $\A$, which by the previous results is of the form $\A:=\mathcal{D}(\Phi',n-2)$ for $\Phi'=\Phi_n+\phi_{n-1}+\phi_{n-2}$.

\begin{pTheoremCY}(Theorem \ref{TheoremCY}) \label{pTheoremCY}
Suppose $\phi_{n-1}=0$, then $\mathcal{A}$ is $n$-CY.
\end{pTheoremCY}

3)
An application to symplectic reflection algebras. Let $H$ be a path algebra with relations Morita equivalent to an undeformed symplectic reflection algebra. So $H$ is $2$-Koszul and $2n$-CY and $H=\mathcal{D}(\Phi_{2n},2n-2)$ for some homogeneous degree $2n$ superpotential $\Phi_{2n}$.

\begin{pSym}(Theorem \ref{Sym}) \label{pSym}
Any PBW deformation of $H$ is a zeroPBW deformation and of the form $\mathcal{D}(\Phi',n-2)$ for $\Phi'=\Phi_{2n}+\phi_{2n-2}$ an inhomogeneous superpotential. Hence any PBW deformation of $H$ is $2n$-CY, and all symplectic reflection algebras are Morita equivalent to $2n$-CY algebras of the form $\mathcal{D}(\Phi',2n-2)$.
\end{pSym}

\subsection{Contents}
We outline the structure of the paper.

Section 2: We discuss preliminaries, listing definitions and results fundamental to the rest of the paper. These concern quivers, superpotentials, CY algebras, $N$-Koszul algebras, and PBW deformations. In particular we recall a Theorem classifying PBW deformations from \cite{BG}, and several Theorems from \cite{BSW} concerning path algebras with relations defined by superpotentials and CY algebras.

Section 3: The main results of the paper are stated and proved here. We begin by making two definitions required to state our theorems, \emph{$k$-coherent} superpotentials and $\emph{zeroPBW}$ PBW deformations. We then prove Theorem \ref{Theorem3} classifying which PBW deformations of an $N$-Koszul, CY, superpotential algebra $\mathcal{D}(\Phi_n,k)$ are of the form $\mathcal{D}(\Phi',k)$ for a inhomogenous superpotential $\Phi'=\Phi_n + \phi_{n-1} + \dots + \phi_k$. We then specialise to the case of 2-Koszul, $n$-CY, $\mathcal{D}(\Phi_n,n-2)$, and consider PBW  deformations of the form $\mathcal{D}(\Phi_n+\phi_{n-2}, n-2)$, proving they are $n$-CY in \ Theorem \ref{TheoremCY}. We note known related results.

Section 4: We apply our results to symplectic reflection algebras. We recall the definition of symplectic reflection algebras, and their classification as PBW deformations by \cite{EG}. We consider the Morita equivalent path algebras with relations, which are of the form $\mathcal{D}(\Phi_{2n}+\phi_{2n-2},2n-2)$ and $2n$-CY by the previous results. We calculate several examples, including the case of preprojective algebras.

Section 5: Finally we give an analysis of PBW deformations of algebras $\mathcal{D}(\Phi_2,0)$ Morita equivalent to $\mathbb{C}[\mathbb{C}^2] \rtimes G$ for $G$ a finite subgroup of $\GL_2(\mathbb{C})$. In particular we show there are no nontrivial PBW deformations when $G$ is a small subgroup not contained in $\SL_2(\mathbb{C})$.

\subsection{Acknowledgments}
The author is an EPSRC funded student at the university of Edinburgh, and this material will form part of his PhD thesis. The author would like to express his thanks to his supervisors, Prof. Iain Gordon and Dr. Michael Wemyss, for much guidance and patience, and also to the EPSRC.

\section{Preliminaries} \label{preliminaries}

In this section we set up the definitions and give a summary of results which we use later.

\subsection{Quivers and Superpotentials} \label{quivers}
Here we give the definition of a quiver, its path algebra, and its path algebra with relations. We define certain elements of the path algebra to be superpotentials, and give a construction to create relations on the quiver from the superpotential. We are following the set up of \cite{BSW}. 
\subsubsection*{Quivers}
\begin{quiver}
A \emph{quiver} is a directed multigraph. We will denote a quiver $Q$ by $Q=(Q_0,Q_1)$, with $Q_0$ the set of vertices and $Q_1$ the set of arrows. The set of arrows is equipped with head and tail maps $h,t:Q_1 \rightarrow Q_0$ which take an arrow to the vertices that are its head and tail respectively.

A \emph{non-trivial path} in the quiver is defined to be a sequence of arrows 
\begin{equation*}
p=a_r \dots a_2 a_1 \qquad \text{with $a_i \in Q_1$ satisfying $h(a_i)=t(a_{i+1})$ for $1 \le i \le r-1$. }
\end{equation*}
\begin{center}
\begin{tikzpicture}
\node (C1) at (0,0)  {$\bullet$};
\node (C2) at (1,0)  {$\bullet$};
\node (C3) at (2,0)  {$\dots$};
\node (C4) at (3,0) {$\bullet$};
\draw[->]  (C1) edge  node[above] {$a_1$} (C2);
\draw[->]  (C2) edge  node[above] {$a_2$} (C3);
\draw[->]  (C3) edge  node[above] {$a_r$} (C4);
\end{tikzpicture}
\end{center}

We will reuse the notation of the head and tail maps for the head and tail of a path, defining $h(p)=h(a_r)$ and $t(p)=t(a_1)$ when $p=a_r \dots a_2 a_1$. We also define a \emph{trivial path} $e_v$ for each vertex $v \in Q_0$, which has both head and tail equal to $v$. A path $p$ is called \emph{closed}  if $h(p)=t(p)$. The \emph{pathlength} of a nontrivial path $p=a_r \dots a_1$, where each $a_i$ is an arrow, is defined to be $r$. A trivial path is defined to have pathlength $0$. We will denote the pathlength of a path $p$ by $|p|$.
\end{quiver}

\begin{pathalg}
Let $\K$ be a field. We define the \emph{path algebra} of the quiver $Q$, $\K Q$, as follows: $\K Q$ as a $\K$-vector space has a basis given by the paths in the quiver; an associative multiplication is defined by concatenation of paths. 
\begin{equation*}
p.q= \left\{
\begin{array}{cl}
pq & \qquad \text{if $h(q)=t(p)$} \\
0 & \qquad \text{otherwise}
\end{array} \right.
\end{equation*}

\end{pathalg}

We define $S$ to be the subalgebra of this generated by the trivial paths, and $V$ to be the $\K$ vector subspace of $\K Q$ spanned by the arrows, $a \in Q_1$. Then $S$ is a semisimple algebra, with one simple module for each vertex.

For an arrow $a$, we have $ a=e_{h(a)}.a.e_{t(a)} $, and so $V$ has the structure of a left $S^e:=S \otimes_{\K} S^{op}$ module and 
$\K Q$ can be identified with the  tensor algebra $T_S(V)=S \oplus V \oplus (V \otimes_S V) \oplus \dots$, equating the path $a_1 \dots a_r$ with $a_1 \otimes_{S} \dots \otimes_S a_r$.

The algebra $T_S(V)=\K Q$ is equipped with a grading and filtration by pathlength, the graded part in degree $n$ is $T_S(V)^n:=V ^{\otimes_{S} n}$, and the filtered part is $F^n(T_S(V)):= V^{\otimes_{S} n} \oplus \dots \oplus V \oplus S$. 

Given $R \subset T_S(V)$ we define $I(R)$ to be the two sided ideal in $T_S(V)$ generated by $R$. We then define 
\begin{equation*}
\frac{\K Q}{R}:=\frac{\K Q}{I(R)}
\end{equation*}
and refer to it as the path algebra with \emph{relations} $R$.

We define $<R>$ to be the $S^e$ submodule of $T_S(V)$ generated by $R$.

\subsubsection*{Superpotentials} We define superpotentials and twisted superpotentials in the homogeneous and  inhomogeneous cases. In the majority of this text we will be working with non-twisted superpotentials, but in Section \ref{General} we will consider the twisted case.
\begin{superpotential}
Let $\Phi_n \in T_S(V)^n$, written $\Phi_n = \Sigma_p c_p p $, where the sum is taken over all paths $p$ of pathlength $n$ with coefficients $c_p$ in $\K$. Our convention will be that $c_0=0$. For example if $t(a) \neq h(b)$ then $ab=0$ and hence $c_{ab}=c_0=0$.

We will say $\Phi_n$ satisfies the $n$ \emph{superpotential condition} if $c_{aq}=(-1)^{n-1} c_{qa}$ for all $a \in Q_1$ and paths $q$, and if $\Phi_n$ satisfies the $n$ superpotential condition we will call it a \emph{homogeneous superpotential of degree $n$}. In particular this requires $c_p=0$ if $p$ is not a closed path.

Let $\sigma \in Aut_{\K}(\K Q)$ be a graded automorphism, so that $\sigma(T_S(V)^n)=T_S(V)^n$ for all $n \ge 0$, and assume that $\sigma(h(p))=h(\sigma(p)), \, \sigma(t(p))=t(\sigma(p))$ for any path $p$. 

We will say $\Phi_n$ satisfies the \emph{twisted $n$ superpotential condition} if $c_{\sigma(a)q}=(-1)^{n-1} c_{qa}$ for $a \in Q_1$ and paths $q$. We will say $\Phi_n$ is a \emph{$\sigma$-twisted homogeneous superpotential of degree $n$} if it satisfies the \emph{twisted $n$ superpotential condition}.
\end{superpotential}

\begin{superpotential2} \label{superpotential2}
Consider an element $\Phi \in F^n(T_S(V))$ written $\Phi=\Sigma_{m \le n} \Sigma_{|p|=m} c_p p$ as above. We will say $\Phi$  is an \emph{inhomogeneous superpotential of degree $n$} if it satisfies the conditions $c_{aq}=(-1)^{n-1} c_{qa}$ for any path $q$ and arrow $a$, and for some $p$ of length $n$ the coefficient $c_p$ is non-zero. An inhomogeneous superpotential $\Phi$ can be written in homogeneous parts $\Phi=\phi_n + \phi_{n-1} + \dots +\phi_0$ where each of the $\phi_m$ satisfies the $n$-superpotential condition and $\phi_n$ is non-zero. Note that if $n$ is even, $m$ odd, and char$\K \neq 2$ then $\phi_m=0$, and also that any $\phi_0 \in S$ satisfies the $n$ superpotential condition.

We will say this is an \emph{$\sigma$-twisted inhomogeneous superpotential of degree $n$} if it satisfies $c_{\sigma(a)q}=(-1)^{n-1} c_{qa}$ for any path $q$ and arrow $a$. A twisted inhomogeneous superpotential $\Phi$ can be written in homogeneous parts $\Phi=\phi_n + \phi_{n-1} + \dots +\phi_0$ where each of the $\phi_m$ satisfies the $n$-twisted superpotential condition.
\end{superpotential2}

\subsubsection*{Differentiation} We define differentiation by paths.

\begin{diff} Let $p$ be a path in $\K Q$, and $a \in Q_1$. We then define left/right derivative of $p$ by $a$ as;
\begin{equation*}
\begin{array}{c c}

\delta_a p=\left\{\begin{array}{cl}
	q, & \qquad \text{if } p=aq \\
	0, & \qquad \text{otherwise}
	   \end{array}\right.

&

p\delta'_a =\left\{\begin{array}{cl}
	q, & \qquad \text{if } p=qa \\
	0, & \qquad \text{otherwise}
	   \end{array}\right.

\end{array}
\end{equation*}

We extend this so that if $q=a_1 \dots a_n$ is a path of length $n$ then $\delta_q= \delta_{a_n} \dots \delta_{a_1}$, and $\delta'_q=\delta'_{a_n} \dots \delta'_{a_1}$. We will call the operation $\delta_q$ left differentiation by the path $q$, and $\delta'_q$ right differentiation by the path $q$.
\end{diff}

Note that for $\Phi$  a degree $n$ inhomogeneous superpotential, $b,c \in Q_1$, and $p$ any path
 \begin{equation*}
\delta_b \Phi = (-1)^{n-1} \Phi \delta'_b \, ,
\end{equation*}
\begin{equation*}
\delta_{b} \Phi \delta'_{c} = (-1)^{n-1} \delta_{b} \delta_{c} \Phi= (-1)^{n-1} \delta_{cb} \Phi \, ,
\end{equation*}
and 
\begin{equation*}
\delta_{p} \Phi = \sum_{a \in Q_1} a \delta_a \delta_p \Phi=\sum_{a\in Q_1} \delta_p \Phi \delta'_a a \, .
\end{equation*}

\begin{diff2} \label{diff2}
Given an inhomogeneous superpotential $\Phi$ of degree $n$ we can consider the $S^e$-modules $\mathcal{W}_{n-k} \subset \K Q$ given by:
\begin{equation*}
\mathcal{W}_{n-k}=<\delta_p \Phi : |p|=k >
\end{equation*}

We then define the algebra $\mathcal{D}(\Phi,k)$ to be
\begin{equation*}
\mathcal{D}(\Phi,k):=\frac{\K Q}{I(\mathcal{W}_{n-k})}
\end{equation*}
We will call this the \emph{superpotential algebra} $\mathcal{D}(\Phi,k)$.
\end{diff2}

\begin{example2} \label{example2} 
Consider the quiver, $Q$, with one vertex, $\bullet$, and two arrows $x$ and $y$. Then $\K Q$ is the free algebra on two elements, $\K <x,y>$\\

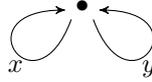
\begin{center}
\begin{tikzpicture} [bend angle=45, looseness=1]
\node (C1) at (0,0)  {$\bullet$};
\node (C1a) at (0.1,-0.05)  {};
\node (C1b) at (-0.1,-0.05)  {};
\draw[<-]  (C1b) edge [in=250,out=180,loop,looseness=24] node[below] {$x$} (C1b);
\draw[<-]  (C1a) edge [in=-70,out=0,loop,looseness=24] node[below] {$y$} (C1a);
\end{tikzpicture}
\end{center}
 We consider some superpotentials on $Q$.
\begin{itemize}
\item[1)] $\Phi_2=xy-yx$ is a degree 2 homogeneous superpotential.
\item[2)] $\Phi=xy-yx -e_{\bullet}$ is an inhomogeneous superpotential of degree 2
\item[3)] $\Phi_3=xyx+xxy+yxx$ is a degree 3 homogeneous superpotential.
\end{itemize}
These give us the algebras
\begin{itemize}
\item[1)] $\mathcal{D}(\Phi_2,0) = \K [x,y]$, as $\mathcal{W}_2=<xy-yx>$ gives relations $R=\{xy-yx=0\}$
\item[2)] $\mathcal{D}(\Phi,0)=A_1(\K)$, the first Weyl algebra, is the quiver with relations given by $R=\{xy-yx=e_{\bullet}\}$
\item[3)] $\mathcal{D}(\Phi_3,1)$ is the quiver with relations given by $R=\{yx+xy=0, xx=0\}$
\end{itemize}
\end{example2}

\subsection{Calabi-Yau algebras}
We define Calabi-Yau algebras following the definitions of \cite{CY} and \cite{AIR}. We recall the definition of self dual used in \cite{BSW}, and note that the existence of a self dual finite projective $A^e$-module resolution of length $n$ implies an algebra is $n$-Calabi-Yau.

Let $A$ be an associative $\K$ algebra, set $A^e=A \otimes_{\K} A^{op}$ and $D(A^e)$ to be the unbounded derived category of left $A^e$ modules. We note that there are two $A^e$-module structures on $A \otimes_{\K} A$, the \emph{inner structure} given by $(a\otimes b)(x\otimes y)=(xb \otimes ay)$, and the \emph{outer structure} given by $(a \otimes b)(x \otimes y)= (ax \otimes yb)$. By considering $A^e$ with the outer structure for any $A^e$ module, $M$, we can make $\Hom_{A^e}(M,A^e)$ an $A^e$-module using the inner structure.

\begin{CY}
Let $n \ge 2$. Then $A$ is $n$ \emph{bimodule Calabi-Yau (n-CY)} if  $ A$ has a finite length resolution by finitely generated projective $A^e$-modules, and
\begin{equation*}
\RHom_{A^e}(A,A^e)[n] \cong A \quad \text{ in $D(A^e)$}
\end{equation*}
\end{CY}

In the paper \cite{BSW} self-duality of a complex is used. We give the definition here. Denote $(-)^{\vee}:=\Hom_{A^e}(-,A^e):A^e$-Mod $\rightarrow A^e$-Mod.

\begin{dual2}
Define a complex of $A^e$ modules, $\mathcal{C}^{\bullet}$, of length $n$ to be self dual, written 
\begin{equation*}
\Hom_{A^e}(\mathcal{C}^{\bullet},A^e) \cong \mathcal{C}^{n-\bullet},
\end{equation*}
if there exist $A^e$-module isomorphisms $\alpha_i$ such that the following diagram commutes:

\begin{center}
\begin{tikzpicture}[node distance=2cm, auto]
\node (P_n) {$C_n$};
\node (P_{n-1}) [right of=P_n] {$C_{n-1}$};
\node(PMid)[right of=P_{n-1}]{\dots};
\node (P_1) [right of=PMid] {$C_1$};
\node (P_0) [right of=P_1] {$C_0$};

\node (Pv_n) [below of=P_n] {$C_0^{\vee}$};
\node (Pv_{n-1}) [right of=Pv_n] {$C_1^{\vee}$};
\node(PMidv)[right of=Pv_{n-1}]{\dots};
\node (Pv_1) [right of=PMidv] {$C_{n-1}^{\vee}$};
\node (Pv_0) [right of=Pv_1] {$C_n^{\vee}$};

\draw[->] (P_n) to node {$d_n$} (P_{n-1});
\draw[->] (P_{n-1}) to node {$d_{n-1}$} (PMid);
\draw[->] (PMid) to node {$d_2$} (P_1);
\draw[->] (P_1) to node {$d_1$} (P_0);

\draw[->] (Pv_n) to node {$-d_1^{\vee}$} (Pv_{n-1});
\draw[->] (Pv_{n-1}) to node {$-d_2^{\vee}$} (PMidv);
\draw[->] (PMidv) to node {$-d_{n-1}^{\vee}$} (Pv_1);
\draw[->] (Pv_1) to node {$-d_n^{\vee}$} (Pv_0);

\draw[->] (P_n) to node {$\alpha_n$} (Pv_n);
\draw[->] (P_{n-1}) to node {$\alpha_{n-1}$} (Pv_{n-1});
\draw[->] (P_1) to node {$\alpha_{1}$} (Pv_1);
\draw[->] (P_0) to node {$\alpha_0$} (Pv_0);
\end{tikzpicture}
\end{center}
\end{dual2}

By construction the existence of a length $n$ self dual projective $A^e$-module resolution of $A$ implies that $A$ is $n$-CY, and we will use this later in Section \ref{Main Results} to show algebras are $n$-CY.

We note some properties of Calabi-Yau algebras
\begin{Calabi2}\label{Calabi2}
Let $A$ be $n$-CY $\K$-algebra, where $\K$ is algebraically closed. Then
\begin{itemize}
\item[1.] If there exists a non-zero finite dimensional $A$-module, then $A$ has global dimension $n$.
\item[2.] For $X,Y \in D(A)$ with finite dimensional total homology
\begin{equation*}
\Hom_{D(A)}(X,Y) \cong \Hom_{D(A)}(Y,X[n])^*
\end{equation*}
where $*$ denotes the standard $\K$ dual.
\end{itemize}
\end{Calabi2}
\begin{proof}
These are some of the standard properties of CY algebras, see for example \cite[Proposition 2.4]{AIR} and\cite[Section 2]{BT}  for proofs.
\end{proof}

\subsection{Koszul Algebras}
Here we give the definitions of $N$-Koszul algebras.

The concept of an $N$-Koszul algebra was introduced by Berger, and is defined and studied in the papers \cite{KNQ},\cite{Hom}, \cite{BG}, and \cite{BMN}. It generalises the concept of a Koszul algebra, which is defined here to be a 2-Koszul algebra. 

\begin{NKoszul} \label{NKoszul}

Let $S$ be a semisimple ring, $V$ a left $S^e$-module, and $T_S(V)$ the tensor algebra. An algebra is \emph{$N$-homogeneous} if it is given in the form $A=\frac{T_S (V)}{I(R)}$ for $R$ an $S^e$-submodule of  $V^{\otimes_S N}$.

We now define the \emph{Koszul $N$-complex}.
For $N$-homogeneous $A=\frac{T_S(V)}{I(R)}$ define
\begin{eqnarray*}
K_i & = & \cap_{j+N+k=i} \left( V^{\otimes j} \otimes_S R \otimes_S V^{\otimes k} \right) \qquad i \ge N \\
K_i & = & V^{\otimes_S  i}  \qquad 0<i< N \\
K_0 & = & S
\end{eqnarray*}
and let
\begin{eqnarray*}
K^i(A)=A \otimes_S K_i \otimes_S A
\end{eqnarray*}

 Now we can define an $N$-complex (In an $N$-complex $d^N=0$ rather than $d^2=0$)
\begin{equation*}
\dots \xrightarrow{d}  K^n (A) \xrightarrow{d} \dots \xrightarrow{d} K^{0} (A) \xrightarrow{\mu} A \rightarrow 0
\end{equation*}
To define the differentials let $\mu$ be multiplication, and $d_l, d_r : A \otimes_S V^{\otimes_S i} \otimes_S A \rightarrow A \otimes_S V^{\otimes_S (i-1)} \otimes_S A$ be defined by
\begin{eqnarray*}
d_l(\alpha \otimes v_1 v_2 \dots v_i \otimes \beta)= \alpha v_1 \otimes v_2 \dots v_i \otimes \beta \\
d_r(\alpha \otimes v_1 \dots v_{i-1} v_i \otimes \beta)= \alpha \otimes v_1 \dots v_{i-1} \otimes v_i \beta 
\end{eqnarray*}
As $R \subset V^{\otimes_S N}$ we see $K_i \subset V^{\otimes_S i}$ and hence can consider the restriction of $d_l,d_r$ to $K^i(A)$.
Now choose $q \in \K$ a primitive $N^{th}$ root of unity, which we need to assume exists in $\K$, and define $d: K^i(A)  \rightarrow  K^{i-1}(A)$ as $d= d_l|_{K^i(A)}-(q)^{i-1} d_r|_{K^i(A)}$.

As $d_l$ and $d_r$ commute with $d_l^N=0$ and $d_r^N=0$ $d$ defines an $N$ differential.

To such an $N$-complex we can associate a complex by contracting several terms together. This is done by splitting the $N$-complex into sections of $N$ consecutive differentials from the right. Each one of these is collapsed to a term with two differentials by keeping the right most differential, and composing the other $N-1$. This defines a complex of the form
\begin{equation*}
\dots \rightarrow  K^{N+1} (A)  \xrightarrow{d} K^{N} \xrightarrow{d^{N-1}} K^{0} (A) \xrightarrow{\mu} A \rightarrow 0
\end{equation*}
which exists regardless of the field $\K$, with $d=(d_l-d_r)|_{K^j(A)}$ and $d^{N-1}=(d_l^{N-1} + d_l^{N-2}d_r + \dots + d_r^{N-2})|_{K^j(A)}$.

We say that $A$ is \emph{$N$-Koszul} if this complex is exact, and hence gives an $A^e$-module resolution of $A$.
\end{NKoszul}

We call $2$-Koszul algebras \emph{Koszul}. In this case the $N$-complex is in fact a complex.

\subsection{Superpotentials and Higher Order Derivations} We summarise some results from the paper \cite{BSW}.  Later we will be particularly interested the applications to skew group algebras so we recall their definition and the relevant results from \cite{BSW}.
\subsubsection*{Superpotentials and Higer Order Derivations} \label{BSW}
Here we state several results from the paper \cite{BSW} concerning  superpotential algebras.
\vspace{0.4cm}

Let $Q$ be a quiver, with path algebra $\mathbb{C} Q$. 
Let $\Phi_n$ be a  homogeneous superpotential  of degree $n$, and $A=\mathcal{D}(\Phi_n,n-2)$. For $i=0, \dots n$ we have the $S^e$-modules $\mathcal{W}_{i}$ as in Definition \ref{diff2} above, and we can define a complex $\mathcal{W}^{\bullet}$  by
\begin{equation} \label{Complex}
0 \rightarrow A \otimes_S \mathcal{W}_n \otimes_S A \xrightarrow{d_n} A \otimes_S \mathcal{W}_{n-1} \otimes_S A \xrightarrow{d_{n-1}} \dots \xrightarrow{d_1} A \otimes_S \mathcal{W}_0 \otimes_S A \rightarrow 0.
\end{equation} 
To define the differential first define $d_i^l, d_i^r: A \otimes_S \mathcal{W}_i \otimes_S A \rightarrow A \otimes_S \mathcal{W}_{i-1} \otimes_S A$ by

\begin{eqnarray*}
d_i^l(\alpha \otimes \delta_p \Phi_n \otimes \beta)&= &\sum_{a \in Q_1} \alpha a \otimes \delta_a \delta_p \Phi_n \otimes \beta \\
d_i^r(\alpha \otimes\delta_p \Phi_n \otimes \beta)&=&\sum_{a\in Q_1} \alpha \otimes \delta_p \Phi_n \delta'_a \otimes a \beta
\end{eqnarray*}
for $p$ any path of length $n-i$.

The differential $d_i$ is defined as 
\begin{equation*}
d_i=\epsilon_i (d_i^l +(-1)^i d_i^r) \text{  where } \epsilon_i := \left\{
\begin{array}{c c}  
 (-1)^{i(n-i)}  & \text{if } i<(n+1)/2 \\
   1  &    \text{otherwise} \end{array}
\right.
\end{equation*} 

This  is a complex as $d^l, d^r$ commute and have square 0. We show this for $d^l$, the other case being similar. Writing $\Phi_n=\sum_{|t|=n} c_t t$, we have 
\begin{eqnarray*}
d^l_{j-1} \circ d^l_{j}(1 \otimes \delta_p \Phi_n \otimes 1) & = & \sum_{a,b} ab \otimes \delta_{ab} \delta_p \Phi_n \otimes 1\\
&=& \sum_{a,b,q} ab \otimes c_{pabq} q \otimes 1 \\
&=& \sum_{a,b,q} (-1)^{(n-1)(j-2)} c_{qpab} ab \otimes q \otimes 1 \\
&=& \sum_q (-1)^{(n-1)j} \delta_{qp} \Phi_n \otimes q \otimes 1 = 0
\end{eqnarray*}
where $a,b \in Q_1$, $p$ is a path of length $(n-j)$, $q$ a path of length $j-2$, and the sums are taken over all such $a,b$ and $q$. 

\begin{BSW1}[{\cite[Section 6]{BSW}}]\label{BSW1}
Let $A=\mathcal{D}(\Phi_n,n-2)$, then $\mathcal{W}^{\bullet}$ is a self dual complex of projective $A^e$-modules.
\end{BSW1}

\begin{BSW2}[{\cite[Theorem 6.2]{BSW}}] \label{BSW2} Let $A=\frac{\mathbb{C} Q}{R}$ be a path algebra with relations. Then A is 2-Koszul and $n$-CY if and only if $A$ is of the form $\mathcal{D}(\Phi_n,n-2)$ for some homogeneous superpotential $\Phi_n$ of degree n and the attached complex $\mathcal{W}^{\bullet}$ is a resolution of $A$. In this case the resolution equals the Koszul complex of Definition \ref{NKoszul} with each $\mathcal{W}_i=K_i$.
\end{BSW2}

The complex (\ref{Complex}) is the relevant complex for $\mathcal{D}(\Phi_n,n-2)$, where the relations are obtained by  differentiation by paths of length $n-2$. More generally, differentiating by paths of length $k$, we need another complex. Let $N=n-k$, and in this case we define an $N$-complex $\mathcal{\widetilde{W}}^{\bullet}$, again making use of the $S^e$-modules $\mathcal{W}_j$.
\begin{equation} \label{Complex 2}
0 \rightarrow A \otimes_S \mathcal{W}_n \otimes_S A \xrightarrow{d_n} A \otimes_S \mathcal{W}_{n-1} \otimes_S A \xrightarrow{d_{n-1}} \dots \xrightarrow{d_1} A \otimes_S \mathcal{W}_0 \otimes_S A \rightarrow 0
\end{equation}
with differential $d_i: A \otimes_S \mathcal{W}_i \otimes_S A \rightarrow A \otimes_S \mathcal{W}_{i-1} \otimes_S A$ defined by $d_i=d_i^l + (q)^i d_i^r$, for $q$ a primitive $N^{th}$ root of unity, which we assume exists in $\K$. This can be contracted into a $2$-complex, $\mathcal{\widehat{W}}^{\bullet}$
\begin{equation*} \label{Complex 3}
0 \rightarrow A \otimes_S \mathcal{W}_{mN+1} \otimes_S A \xrightarrow{d} A \otimes_S \mathcal{W}_{mN} \otimes_S A \xrightarrow{d^{N}} A \otimes_S \mathcal{W}_{(m-1)N+1} \otimes_S A \xrightarrow{d} \dots \xrightarrow{d} A \otimes_S \mathcal{W}_0 \otimes_S A \rightarrow 0
\end{equation*}
by composing the differentials in $\mathcal{\widetilde{W}}^{\bullet}$, as in Definition \ref{NKoszul}, where $m$ is the largest integer such that $m \le n/N$.

\begin{BSWK}[{\cite[Theorem 6.8]{BSW}}]\label{BSWK}
Let $A=\frac{\mathbb{C} Q}{R}$ be a path algebra with relations. Then $A$ is $(n-k)$-Koszul and $(k+2)$-Calabi Yau if and only if it is of the form $\mathcal{D}(\Phi_n,k)$  where $\Phi_n$ is an homogeneous superpotential of degree $n$ and $\mathcal{\widehat{W}}^{\bullet}$ is a resolution of $A$. In this case $\mathcal{\widehat{W}}^{\bullet}$ equals the Koszul $(n-k)$-complex as in Definition \ref{NKoszul}. 
\end{BSWK}

\subsubsection*{Skew Group Algebras}
Let $G$ be a finite subgroup of $\GL(V)$.  We can form the smash product $\K[V] \rtimes G$ which is the semi direct product with the action of $G$ given by that of $\GL(V)$. The multiplication of $(f_1,g_1), (f_2,g_2) \in \K[V] \rtimes G$ is given by
\begin{equation*}
(f_1,g_1).(f_2,g_2)=(f_1 f_2^{g_1}, \, g_1 g_2)
\end{equation*}

This is a graded algebra with $G$ in degree 0 and $V^*$ in degree 1, where $\K[V]=Sym(V^*)$.

For such a group $G$ and representation $V$ we can construct the McKay quiver.  We now assume our field is has characteristic not dividing the order of $G$ and is algebraically closed. The McKay quiver for $(G,V)$ has
a vertex, $i$, for each irreducible representation, $V_i$, of $G$ over $\K$, and $dim_{\K} \Hom(V_i \otimes_{\K} V, V_j)$ arrows from $i$ to $j$. 

\begin{BSW3}[{\cite[Theorem 3.2]{BSW}}]\label{BSW3}
Let $V$ be a $n$-dimensional $\mathbb{C}$-vector space and $G$ be a finite subgroup of $GL(V)$. Then $\mathbb{C}[V] \rtimes G$ is Morita equivalent to $\mathcal{D}(\Phi_n,n-2)$ for some homogeneous superpotential $\Phi_n$ of degree $n$ attached to the McKay quiver $(G,V)$. Moreover $\mathcal{D}(\Phi_n,n-2)$ is $n$-CY and Koszul for $G \le \SL(V)$.
\end{BSW3}

There is a recipe to compute a superpotential $\Phi_n$ such that $\mathbb{C}[V] \rtimes G$ is Morita equivalent to the superpotential algebra $\mathcal{D}(\Phi_n,n-2)$  attached to the McKay quiver $(G,V)$. The recipe is given in \cite[Theorem 3.2]{BSW}. When $G$ is abelian the superpotential algebra $\mathcal{D}(\Phi_n,n-2)$  is in fact isomorphic to $\mathbb{C}[V] \rtimes G$.

\subsection{PBW deformations} \label{PBW deformations}
In this subsection we will recall the definition of the PBW deformations of a graded algebra $A$. We will be considering PBW deformations in the case $A$ an $N$-Koszul $S^e$-module, relative to $S$ a semisimple algebra, and will make use of the setup and results of \cite{BG}.
\vspace{4mm}

Let $S$ be a semisimple algebra, $V$ a left $S^e$-module, and $T_S(V)$ the tensor algebra of $V$ over $S$. Consider the ($\mathbb{Z}_{\ge0}$) grading on this with degree $n$ part $T_S(V)_n=V^{\otimes_S n}$, and filtered parts $F^n= V^{\otimes_S n} \oplus \dots \oplus V \oplus S$. For $R$ an $S^e$-submodule of $V^{\otimes_S N}$ define $ I(R)=\sum_{i,j \ge 0} V^{\otimes i} R V^{\otimes j}$ to be the two sided ideal in $T_S(V)$ generated by $R$, which is graded by $I(R)_n=I(R) \cap V^{\otimes_S n}$. Define $A:=\frac{T_S(V)}{R}:=\frac{T_S(V)}{I(R)}$, and as $R$ is homogeneous this is a graded algebra with degree $n$ part
$
A_n:=\frac{V^{\otimes_S n}}{I(R)_n}$.

Now define the projection map $\pi_N: F^N \rightarrow V^{\otimes_S N}$, and  let $P$ be an $S^e$-submodule of $F^N$ such that $\pi(P)=R$.  Let $I(P)$ be the 2 sided ideal in $T_S(V)$ generated by $P$. This is not graded but is filtered by $I(P)^n = I(P) \cap F^n$, and hence $\mathcal{A}=\frac{T_S(V)}{I(P)}$ is also a filtered algebra with $\mathcal{A}^n=\frac{F^n}{I(P)^n}$.

We can construct the associated graded algebra $gr(\mathcal{A})$, which is graded with $gr(\mathcal{A})_n=\frac{\mathcal{A}^n}{\mathcal{A}^{n-1}}$. Identify $gr(\mathcal{A})_n$ with the $S^e$-module $\frac{F^n}{I(P)^n + F^{n-1}}$ by noting $I(P)^n \cap F^{n-1}=I(P)^{n-1}$, and consider the maps $\phi_n: V^{\otimes n} \rightarrow gr(\mathcal{A})_n$ defined as the composition $V^{\otimes n} \hookrightarrow F^n \rightarrow \frac{F^n}{I(P)^n + F^{n-1}}$. This allows us to define a surjective algebra morphism $\phi=\bigoplus_{n \ge 0} \phi_n : T_S(V) \rightarrow gr(\mathcal{A})$. Now $\phi_N(R)=0$ as $P+F^{N-1}=R \oplus F^{N-1}$, and hence this defines a surjective morphism of $\mathbb{Z}_{\ge 0}$-graded $S^e$-modules 
\begin{equation*}
p:A \rightarrow gr(\mathcal{A})
\end{equation*}

\begin{PBW} \label{PBWdef}
We say that $\mathcal{A}=\frac{T_S(V)}{P}$ is a \emph{PBW deformation} of $A=\frac{T_S(V)}{R}$ if $p$ is an isomorphism. We say that $P$ is of \emph{PBW type} if $\mathcal{A}$ is a PBW deformation of $A$.
\end{PBW}

The papers \cite{BT} and \cite{BG} prove a collection of conditions on $P$ equivalent to it being PBW type which we state here.

First note that for $P\subset F^N$ to be of PBW type it must be the case that $P\cap F^{N-1}=\{0\}$. Hence any PBW type $P$ can be given as $P=\{r-\theta (r) : r \in R \}$  for an $S^e$-module map  $\theta: R \rightarrow F^{N-1}$. Such a map can be written in homogeneous components as $\theta= \theta_{N-1} + \dots +\theta_0$, with $\theta_j : R \rightarrow V^{\otimes_S j}$.

Consider the map $\psi(\theta_j)$ defined for each $\theta_j$ as
\begin{eqnarray*}
\psi(\theta_j) := id \otimes \theta_j - \theta_j \otimes id : (V \otimes_{S} R) \cap (R \otimes_{S} V) \rightarrow V^{\otimes_{S} (j+1)}
\end{eqnarray*}
where $(V \otimes_{S} R) \cap (R \otimes_{S} V) \subset V^{\otimes_S (N+1)}$

\begin{BT1}[{\cite[Section 3]{BG}}] \label{BT1}
The $S^e$-module $P$ is of PBW type  if and only if the following conditions are satisfied
\begin{itemize}
\item[PBW1)] $P \cap F^{N-1}= \{ 0 \} $ 
\item[PBW2)] $\Im \, \psi(\theta_{N-1})\subset R$  
\item[PBW3)]  $\Im \, (\theta_j(\psi(\theta_{N-1})) + \psi(\theta_{j-1}) )=\{ 0\}$       for all $1 \le j \le N-1$
\item[PBW4)] $\Im \, \theta_0(\psi(\theta_{N-1}))= \{0\} $
\end{itemize}
\end{BT1}

\begin{examplePBW} \label{examplePBW}
Let $A=\mathcal{D}(\Phi_n,k)$ be a $N$-Koszul $(k+2)$-CY superpotential algebra, where $N=n-k$. Then Theorem \ref{BSWK} tells us that the Koszul $N$-complex equals the complex $\widehat{\mathcal{W}^{\bullet}}$, and $\mathcal{W}_i = K_i$, of Definition \ref{NKoszul}. Hence $ (V \otimes_{S} R) \cap (R \otimes_{S} V) = \mathcal{W}_{N+1}$, and $R=\mathcal{W}_{N}$. We can also use the expression $\delta_p\Phi_n = \sum_{a \in Q_1} a \otimes \delta_a \delta_p \Phi_n = \sum_{a \in Q_1} \delta_p \Phi_n \delta'_a \otimes a $ to give an explicit form for $\psi(\theta_j)$ with $j=0, \dots, N-1$:
\begin{eqnarray*}
\psi(\theta_j): \mathcal{W}_{N+1} &\rightarrow & V^{\otimes_S (j+1)} \\
         \delta_p \Phi & \mapsto & \sum_{a \in Q_1} a \otimes  \theta_j(\delta_a\delta_{p} \Phi) - \theta_j(\delta_p \Phi \delta'_a) \otimes a \quad \text{ where $|p|=k-1$}
\end{eqnarray*}
\end{examplePBW}

We note that in the case of $S$ being a field, and $A$ being 2-Koszul, Theorem \ref{BT1} was proved by Braverman and Gaitgory, \cite{BrG}. They also show any PBW deformation gives a graded deformation. With the notation as above we define a graded deformation of $A$,  $A_t$, to be a graded $\K[t]$ algebra with $t$ in degree 1. That is $A_t$ is free as a module over $\K [t]$ and there is an isomorphism $A_t/tA_t \rightarrow A$ . It is shown that for a PBW deformation $\mathcal{A}$ there is a graded deformation of $A$ with fibre at $t=1$ canonically isomorphic to $\mathcal{A}$, \cite[Theorem 4.1]{BrG}.\\

In the case with $\mathcal{D}(\Phi_n,n-2)$ Koszul and $S$ a field there is also the following theorem of Wu and Zhu, \cite{WZ},  which classifies when PBW deformations of a Noetherian $n$-CY algebra are also $n$-CY. In our language this  applies when there is only a single vertex in the quiver.

\begin{WZ} \label{WZ}
Let $\Phi_n$ be a homogeneous superpotential on a single vertex quiver. Let $A=\mathcal{D}(\Phi_n,n-2)$.
If $A$ is $n$-CY Noetherian and Koszul and $\mathcal{A}$ is a PBW deformation given by $\theta_0,\theta_1$ then $\mathcal{A}$ is $n$-CY if and only if
\begin{equation*}
\sum^{n-2}_{i=0} (-1)^i id^{\otimes i} \otimes \theta_1 \otimes id^{\otimes(n-2-i)}(\Phi_n)=0
\end{equation*}
\end{WZ}
\begin{proof}
 See \cite[Theorem 3.1]{WZ}. This proves such a  PBW deformation is $n$-CY if and only if
\begin{eqnarray*}
\sum^{n-2}_{i=0} (-1)^i id^{\otimes i} \otimes \theta_1 \otimes id^{\otimes(n-2-i)}: \cap_i V ^{\otimes i} \otimes R \otimes V^{\otimes (n-2-i)} \rightarrow V^{\otimes n-1}
\end{eqnarray*}
is 0. Note that $\cap_i V ^{\otimes i} \otimes R \otimes V^{\otimes (n-2-i)}=\mathcal{W}_n=<\Phi_n>$ in this case, giving the result.
\end{proof}

\section{Main results} \label{Main Results}
We state and prove our main results concerning the PBW deformations of an $(k+2)$-CY, $(n-k)$-Koszul, superpotential algebra $A:=\mathcal{D}(\Phi_n,k)$, where $\Phi_n$ is a homogeneous superpotential. We will prove a condition for a PBW deformation to be of the form $\mathcal{D}(\Phi',k)$ for an inhomogeneous superpotential $\Phi'=\Phi_n+ \phi_{n-1}+ \dots + \phi_k$ , and prove that inhomogeneous superpotentials with degree $n$ part $\Phi_n$ define PBW deformations of $A$. In the Koszul case we will show that certain PBW deformations of such an $n$-CY algebra are also $n$-CY. 

Throughout this section consider $\Phi_n\in \mathbb{C}Q$ to be a homogeneous superpotential of degree $n$ on some quiver $Q$. This will have inhomogenous superpotentials associated to it, denoted $\Phi'=\Phi_n +\phi_{n-1} + \dots +\phi_k$. 

We introduce two new definitions we make use of in the proof:
\begin{coh} \label{coherent}
We will call $\Phi'$, an inhomogeneous superpotential, \emph{$k$-coherent} if for any $\lambda_p \in \mathbb{C}$
\begin{equation*} 
\sum_{|p|=k} \lambda_p \delta_p \Phi_n=0 \in \mathbb{C} Q \Rightarrow \sum_{|p|=k} \lambda_p \delta_p \Phi'=0 \in \mathbb{C} Q 
\end{equation*}
\end{coh}

\begin{superp} \label{superPBW} 
Let $\A=\frac{T_S(V)}{P}$ and $A=\frac{T_S(V)}{R}$ be as in Definition \ref{PBWdef}. We will say that $\A$ is a \emph{zeroPBW deformation} of $A$ if it is a PBW deformation which also satisfies the \emph{zeroPBW condition}
\begin{equation*}
 \Im(\psi(\theta_{N-1}))=\{0\}.
\end{equation*}
where $\psi$ and $\theta$ are defined as in Section \ref{PBW deformations}.

We will say $P$ is of \emph{zeroPBW type} if $\A$ is a zeroPBW deformation of $A$.
\end{superp}

\begin{zeroPBW} \label{zeroPBW}
An $S^e$ module $P$ is of zeroPBW type if and only if
\begin{itemize}
\item[PBW1)] $P \cap F^{N-1} = \{ 0 \}$
\item[ZPBW)] $\Im(\psi(\theta_j))=\{ 0 \}$     for $j= 0 \dots N-1$.
\end{itemize}
\end{zeroPBW}

\begin{proof}
By definition $P$ is of zeroPBW type if and only if it is of PBW type and also satisfies the zeroPBW condition. It is of PBW type if and only if it satisfies conditions  PBW1,2,3,4) of Theorem \ref{BT1}. Satisfying the the zeroPBW condition is equivalent to reducing the conditions PBW2,3,4) to the condition ZPBW).
\end{proof}

\subsection{Deformations of superpotential algebras} We state and prove our results relating superpotentials and PBW deformations.

\begin{Theorem3}\label{Theorem3}
Let $A=\frac{\mathbb{C}Q}{R}$ be a path algebra with relations which is $(n-k)$-Koszul and $(k+2)$-CY. Then $A=\mathcal{D}(\Phi_n,k)$, for $\Phi_n$ a homogeneous superpotential of degree n, and the zeroPBW deformations of $A$ correspond exactly to the algebras $\mathcal{D}(\Phi',k)$ defined by $k$-coherent inhomogeneous superpotentials of the form $\Phi'=\Phi_n +\phi_{n-1} + \dots + \phi_k$.
\end{Theorem3}

\begin{proof}
As $A$ is $(k+2)$-CY and $(n-k)$-Koszul by Theorem \ref{BSWK} $A=\mathcal{D}(\Phi_n,k)$ for $\Phi_n$ a degree $n$ homogeneous superpotential.

We define inverse maps between $P$ of zeroPBW type and coherent superpotentials $\Phi'$ which will give us the correspondence. We note that zeroPBW deformations exactly correspond to $P$ of zeroPBW type.\\

We first define a map, $\mathcal{F}$, taking $P$ of zeroPBW type to $k$-coherent superpotentials. We define $\mathcal{F}(P)$ as an element of $\mathbb{C}Q$, then show it is $k$-coherent and a superpotential.

Any zeroPBW deformation defined by $P$ is given by a map $\theta=\theta_{n-k-1} + \dots + \theta_0$ with $\theta_j:R \rightarrow V^{\otimes_S j}$ such that $P$ and $\theta$ satisfy the conditions PBW1) and ZPBW) of Lemma \ref{zeroPBW}. We construct $\mathcal{F}(P):=\Phi_n + \phi_{n-1}(\theta_{n-k-1}) + \dots + \phi_{k}(\theta_0)$ by defining

\begin{equation*}
\phi_{n-j}(\theta_{n-k-j}):=-\sum_{|p|=k} p \theta_{n-k-j}(\delta_p \Phi_n)=: \sum_{|s|=n-j}c_s s
\end{equation*}
and show that this is a $k$-coherent superpotential. We note that by this definition $\delta_p \phi_{n-j}(\theta_{n-k-j})=-\theta_{n-k-j}(\delta_p \Phi_n)$ for any $|p|=k$.

Firstly $\mathcal{F}(P)$ is not $k$-coherent precisely when there exist coefficients $\lambda_p \in \mathbb{C}$ such that $\sum_{|p|=k} \lambda_p \delta_p \Phi_n =0$ with $\omega=\sum_{|p|=k} \lambda_p \delta_p \Phi'(\theta) \neq 0$. But this only occurs when there exists $\omega \in P \cap F^{n-1}$ which is non zero. But as $P$ is of zeroPBW type PBW1) holds, and $P \cap F^{n-1} = \{0\}$. Hence $\mathcal{F}(P)$ is $k$-coherent.

Now we show that $\mathcal{F}(P)$ is a superpotential. As $P$ is of zeroPBW type by Lemma \ref{zeroPBW} $\Im \psi (\theta_j)=\{ 0 \}$ for $j=0, \dots , n-k-1$ .  We evaluate $\psi(\theta_j)$ on elements of the form $\delta_q \Phi_n$ where $|q|=k-1$, as in Example \ref{examplePBW},  and use $\theta_j(\delta_p \Phi_n) = - \delta_p \phi_{k+j}(\theta_j)$ to deduce that $\mathcal{F}(P)$ is a superpotential.

\begin{eqnarray*}
0=\psi(\theta_j)(\delta_q \Phi_n) &=& \sum_{a} \big( a \theta_j(\delta_{q a}\Phi_n) +(-1)^{n} \theta_j(\delta_{a q} \Phi_n) a \big) \\
&=& -\sum_{a} \big( a \delta_{qa} \phi_{k+j}(\theta_j) +(-1)^{n} \delta_{aq} \phi_{k+j}(\theta_j) a \big)\\
 & = & -\sum_{a,p} \big( c_{q a p} a p +(-1)^{n} c_{a q p}  p a \big)\\
& = & -\sum_{a,p} \big( (c_{q a p_1 \dots p_j } + (-1)^{n} c_{p_j q p_1 \dots p_{j-1}})ap \big)
\end{eqnarray*}
with the sums over all $a \in Q_1$ and $p=p_1 \dots p_j$ of length $j$.

Hence considering the coefficient of a path $ap$ we see $\phi_{k+j}=\sum_s c_s s$ satisfies the  $n$ superpotential condition
\begin{equation*}
c_{q a p_1 \dots p_j} = (-1)^{n-1} c_{p_j q a p_1 \dots p_{j-1}}
\end{equation*}

Hence each $\phi_{j+k}$ satisfies the $n$ superpotential condition, and $\mathcal{F}(P)$ is a inhomogenous superpotential of degree $n$.

We note that the zeroPBW deformation defined by $P$ is, by construction, $\mathcal{D}(\mathcal{F}(P),k)$ as $P=< \{ r-\theta (r) : r \in R \} >= <\{ \delta_p \mathcal{F}(P) : |p|=k \}>$.\\

Now we define a map, $\mathcal{G}$, sending $k$-coherent superpotentials $\Phi'=\Phi_n + \phi_{n-1}+ \dots +\phi_{n-k}$ to $P$ of zeroPBW type. We first define $\mathcal{G}(\Phi')$ as an $S^e$-module, and define a map $\theta^{\Phi'}$ as in Lemma \ref{zeroPBW}. We show that $\mathcal{F}(\Phi')$ satisfies PBW1), hence $\theta^{\Phi'}$ is well defined. We then show that $\mathcal{G}(\Phi')$ is of zeroPBW type.

This map is defined by $\mathcal{G}(\Phi'):=< \{ \delta_p \Phi' : |p|=k \} >$. We also define maps $\theta^{\Phi'}$,  by $\theta^{\Phi'}:=\sum_{j=0}^{n-k-1} \theta_j^{\phi_{j+k}}$, with 
\begin{eqnarray*}
\theta_j^{\phi_{j+k}}:R & \rightarrow &  V^{\otimes_S j} \\
                \delta_p \Phi_n  & \mapsto & -\delta_p\phi_{j+k}
\end{eqnarray*}
such that $\mathcal{G}(\Phi')=< \{ r-\theta^{\Phi'}(r) : r \in R \}>$, and $\theta_j^{\phi_{j+k}}(\delta_p \Phi_n)=-\delta_p\phi_{j+k}$ for any $|p|=k$.

Firstly we check $\mathcal{G}(\Phi')$ satisfies PBW1), and hence $\theta^{\Phi'}$ is a well defined function. Indeed as $\Phi'$ is $k$-coherent there are no $\omega=\sum_{|p|=k} \lambda_p \delta_p \Phi' \neq 0$, with $\lambda_p \in \mathbb{C}$,  satisfying $\sum_{|p|=k} \lambda_p \delta_p \Phi_n =0$. This implies $P \cap F^{n-1}=\{0\}$, hence $P$ satisfies PBW1).

We next show that $\mathcal{G}(\Phi')$ is of zeroPBW type. By Lemma \ref{zeroPBW}, to show that $\mathcal{G}(\Phi')$ is of zeroPBW type it is enough to show that $\mathcal{G}(\Phi')$ and $\theta^{\Phi'}$ satisfy the conditions PBW1) and ZPBW). We have already shown PBW1) is satisfied, so need only check ZPBW); $\Im \psi (\theta_j^{\phi_{j+k}})=\{ 0 \}$ for $j=0, \dots , n-k-1$. Hence we calculate $\psi(\theta_j^{\phi_{j+k}})(\delta_q \Phi_n)$, for $j=0, \dots, n-k-1$, where $|q|=k-1$.

\begin{eqnarray*}
\psi(\theta_j^{\phi_{j+k}})(\delta_q \Phi_n) &=& \sum_{a} \big( a \theta_j^{\phi_{j+k}}(\delta_{q a}\Phi_n) +(-1)^{n} \theta_j^{\phi_{j+k}}(\delta_{a q} \Phi_n) a \big) \\
&=& -\sum_{a} \big(  a \delta_{qa} \phi_{j+k} +(-1)^{n} \delta_{aq} \phi_{j+k} a \big) \\
 & = & -\sum_{a,p} \big( c_{q a p} a p +(-1)^{n} c_{a q p}  p a \big) \\
& = & -\sum_{a,p} \big( (c_{q a p_1 \dots p_j } + (-1)^{n} c_{p_j q p_1 \dots p_{j-1}})ap \big)
\end{eqnarray*}
where we write  $\phi_j = \sum_{|t|=j} c_t t$. and the sums are over all $a \in Q_1$ and $p=p_1 \dots p_j$ of length $j$.
Hence as $\phi_j$ satisfies the $n$ superpotential condition
\begin{equation*}
c_{q a p_1 \dots p_r} = (-1)^{n-1} c_{p_r q a p_1 \dots p_{r-1}}
\end{equation*}
then $\psi(\theta_j^{\phi_{j+k}})(\delta_q \Phi_n)=0$ for all $q$ and $j$. Therefore $\Im \psi(\theta_j^{\phi_{k+j}})=\{ 0 \}$ for all $j$, and $\mathcal{G}(\Phi')$ is of zeroPBW type.

We note that as $\mathcal{G}(\Phi')=<\{\delta_p \Phi' : |p|=k \}>$ by construction the zeroPBW deformation defined by $\mathcal{G}(\Phi')$ is indeed $\mathcal{D}(\Phi', k)$.\\

Now the maps $\mathcal{F}$ and $\mathcal{G}$ are inverses by construction, and provide a bijection between zeroPBW deformations and $k$-coherent superpotentials of the form $\Phi_n + \phi_{n-1}+ \dots + \phi_k$. Moreover the zeroPBW deformation defined by $P$ equals $\mathcal{D}(\mathcal{F}(P),k)$, and the zeroPBW deformation defined by $\mathcal{G}(\Phi')$ equals $\mathcal{D}(\Phi',k)$. Hence the theorem is proved.
\end{proof}

\begin{Cor1} \label{Cor1}
Let $A=\mathcal{D}(\Phi_n, n-2)$  be a $2$-Koszul $n$-CY algebra, and $\mathcal{A}=\mathcal{D}(\Phi',n-2)$, as in Theorem \ref{Theorem3},  be a zeroPBW deformation of  defined by $\theta_1, \theta_0$. Then if $n$ is even $\theta_1=0$.
\end{Cor1}
\begin{proof}
Note that if $n$ is even then any $\phi_{n-1}$ satisfying the $n$ superpotential condition is 0, as observed after Definition \ref{superpotential2}. In particular a zeroPBW deformation defined by $\theta_0,\theta_1$ corresponds to a superpotential $\Phi_n+\phi_{n-1}+\phi_{n-2}$ by the bijection of Theorem \ref{Theorem3}, and superpotentials with $\phi_{n-1}=0$ correspond to zeroPBW deformations with $\theta_1=0$.  Hence for even $n$ we have $\phi_{n-1}=0$ and so $\theta_1=0$.
\end{proof}

\subsection{CY property of deformations}
We now consider when zeroPBW deformations are CY. We consider the case where $A=\mathcal{D}(\Phi_n,n-2)$ is $n$-CY and $2$-Koszul. We currently prove a result only in the case of a zeroPBW deformation with $\theta_1=0$, which - by Remark \ref{Cor1} - covers all even dimensional cases.

We then briefly mention two results concerning PBW deformations of CY algebras. Our results are weaker, but in a more general setting. One is the result of \cite[Theorem 3.1]{WZ}, which is quoted as Theorem \ref{WZ} above, giving a complete characterisation of CY PBW deformations of a Noetherian CY algebra over a field. The other is the paper \cite{BT} which proves that the zeroPBW deformations of a $3$-CY superpotential algebra are precisely the $3$-CY PBW deformations.

\begin{TheoremCY} \label{TheoremCY}
Let $\mathcal{A}$ be a zeroPBW deformation of $A$ with $\theta_1=0$. Then $\mathcal{A}$ is $n$-CY.
\end{TheoremCY}

\begin{proof}

By Theorem \ref{Theorem3} such a deformation is given by a superpotential $\Phi'=\Phi_n+\phi_{n-2}=\sum_t c_t t$ and $\mathcal{A}=\mathcal{D}(\Phi',n-2)$. We construct a resolution for $\mathcal{A}$ and show it is self-dual.  We recall the complex $\mathcal{W}^{\bullet}$, defined as (\ref{Complex}) in Section \ref{BSW}, and, by Theorem \ref{BSW1}, that $\mathcal{W}^{\bullet}$ is an $A^e$-module resolution of $A$. We then define a complex $\mathcal{W}_{\mathcal{A}}^{\bullet}$
\begin{equation*}
0 \rightarrow \A \otimes_S \mathcal{W}_n \otimes_S \A \xrightarrow{d_n} \A \otimes_S \mathcal{W}_{n-1} \otimes_S \A \xrightarrow{d_{n-1}} \dots \xrightarrow{d_1} \A \otimes_S \mathcal{W}_0 \otimes_S \A \xrightarrow{\mu} \A \rightarrow 0.
\end{equation*}
The $S^e$-modules $\mathcal{W}_k$, and differential $d$, are defined as in Section \ref{BSW}, i.e. $\mathcal{W}_j=< \{ \delta_p \Phi_n : |p|=n-j \}>$, and  $d_i=\epsilon_i (d_i^l +(-1)^i d_i^r)$ where 
\begin{equation*}
\epsilon_i := \left\{
\begin{array}{c c}  
 (-1)^{i(n-i)}  & \text{if } i<(n+1)/2 \\
   1  &    \text{otherwise} \end{array}
\right. 
\end{equation*}
and  $d_i^l, d_i^r: \A \otimes_S \mathcal{W}_i \otimes_S \A \rightarrow \A \otimes_S \mathcal{W}_{i-1} \otimes_S \A$ are defined by

\begin{eqnarray*}
d_i^l(\alpha \otimes \delta_p \Phi_n \otimes \beta)&= &\sum_{a \in Q_1} \alpha a \otimes \delta_a \delta_p \Phi_n \otimes \beta \\
d_i^r(\alpha \otimes\delta_p \Phi_n \otimes \beta)&=&\sum_{a\in Q_1} \alpha \otimes \delta_p \Phi_n \delta'_a \otimes a \beta
\end{eqnarray*}
for $p$ any path of length $n-i$. We will show this is a self-dual resolution of $\A$.

We first check this is a complex, checking $d_{j-1} \circ d_j=0$ for $j=2, \dots, n$ and that $\mu \circ d_1=0$. In the calculations $|p|=n-j$, and the sums are taken over all $a,b \in Q_1$ and paths $q$ of length $j-2$.
\begin{eqnarray*}
d_{j-1} \circ d_{j}(1 \otimes \delta_p \Phi_n \otimes 1) & = & \epsilon_j \epsilon_{j-1} \sum_{a, b}( ab \otimes \delta_{pab} \Phi_n \otimes 1 - 1 \otimes \delta_p \Phi_n \delta'_{ab} \otimes ab) \\
&=& \epsilon_j \epsilon_{j-1}  \sum_{a, b, q}( c_{pabq} ab \otimes q \otimes 1 - 1 \otimes q \otimes c_{pqab} ab) \\
&=& \epsilon_j \epsilon_{j-1}  \sum_{a, b, q}((-1)^{(n-1)j} c_{qpab} ab \otimes q  \otimes 1 - 1 \otimes q \otimes c_{pqab} ab) \\
&=& \epsilon_j \epsilon_{j-1}  \sum_{q}((-1)^{(n-1)j} \delta_{qp}\Phi_n \otimes q  \otimes 1 - 1 \otimes q \otimes \delta_{pq} \Phi_n) \\
&=& -\epsilon_j \epsilon_{j-1}  \sum_{q}((-1)^{(n-1)j} \delta_{qp}\phi_{n-2} \otimes q  \otimes 1 - 1 \otimes q \otimes \delta_{pq} \phi_{n-2}) \\
&=& -\epsilon_j \epsilon_{j-1}  \sum_{q}((-1)^{(n-1)j} c_{qp} 1 \otimes q  \otimes 1 - 1 \otimes q \otimes 1 c_{pq}) \\
&=& -\epsilon_j \epsilon_{j-1}  \sum_{q}(c_{pq} 1 \otimes q  \otimes 1 - 1 \otimes q \otimes 1 c_{pq}) =0
\end{eqnarray*}
\begin{eqnarray*}
\mu \circ d_1 (1 \otimes \delta_p \Phi_n \otimes 1) &=& \epsilon_1\mu(\sum_a a \otimes \delta_pa \Phi_n \otimes 1 - 1 \otimes \delta_{p} \Phi_n \delta_{a} \otimes a )\\
&=& \epsilon_1 \sum_a (c_{pa} a - (-1)^{n-1} c_{ap} a )= \epsilon_1 \sum_a (c_{pa} a - c_{pa} a) = 0
\end{eqnarray*}

Since $\A$ is a PBW deformation of $A$ we have $gr\A \cong A$ and so $gr \mathcal{W}^{\bullet}_{\A} \cong \mathcal{W}^{\bullet}$, where we use the product filtration, $F^n(\A\otimes_S \mathcal{W}_j \otimes_S \A) := \sum_{i+j+k \le n} F^i(\A) \otimes_S \mathcal{W}_j \otimes_S F^k(\A)$. Since $\mathcal{W}^{\bullet}$ is exact, so is $\mathcal{W}^{\bullet}_{\A}$ ,\cite[Lemma 1]{FMG}, and since $\mathcal{W}^{\bullet}$ consists of finitely generated projectives so does $\mathcal{W}_{\A}^{\bullet}$ ,\cite[Lemmas 6.11, 6.16]{MCR}. Hence $\mathcal{W}^{\bullet}_{\A}$ gives a resolution of $\A$ by finitely generated projectives. Now we need only check that this is still self dual, implying $\mathcal{A}$ is $n$-CY.

We now construct isomorphisms $\alpha_k$ between the complex $\mathcal{W}_{\A}^\bullet$ and its dual

\begin{center}
\begin{tikzpicture}[auto]
\node (P_n) at (0,0) {$\A \otimes_S \mathcal{W}_n \otimes_S \A$};
\node (P_{n-1}) at (3.5,0) {$\A \otimes_S \mathcal{W}_{n-1} \otimes_S \A$};
\node(PMid) at (6,0) {\dots};
\node (P_1) at (8.5,0) {$\A \otimes_S \mathcal{W}_{1} \otimes_S \A$};
\node (P_0) at (12,0) {$\A \otimes_S \mathcal{W}_{0} \otimes_S \A$};

\node (Pv_n) at (0,-3) {$(\A \otimes_S \mathcal{W}_{0} \otimes_S \A)^{\vee}$};
\node (Pv_{n-1}) at (3.5,-3) {$(\A \otimes_S \mathcal{W}_{1} \otimes_S \A)^{\vee}$};
\node(PMidv) at (6,-3) {\dots};
\node (Pv_1) at (8.5,-3) {$(\A \otimes_S \mathcal{W}_{n-1} \otimes_S \A)^{\vee}$};
\node (Pv_0) at (12,-3) {$(\A \otimes_S \mathcal{W}_{n} \otimes_S \A)^{\vee}$};

\draw[->] (P_n) to node {$d_n$} (P_{n-1});
\draw[->] (P_{n-1}) to node {$d_{n-1}$} (PMid);
\draw[->] (PMid) to node {$d_2$} (P_1);
\draw[->] (P_1) to node {$d_1$} (P_0);

\draw[->] (Pv_n) to node {$-d_1^{\vee}$} (Pv_{n-1});
\draw[->] (Pv_{n-1}) to node {$-d_2^{\vee}$} (PMidv);
\draw[->] (PMidv) to node {$-d_{n-1}^{\vee}$} (Pv_1);
\draw[->] (Pv_1) to node {$-d_n^{\vee}$} (Pv_0);

\draw[->] (P_n) to node {$\alpha_n$} (Pv_n);
\draw[->] (P_{n-1}) to node {$\alpha_{n-1}$} (Pv_{n-1});
\draw[->] (P_1) to node {$\alpha_{1}$} (Pv_1);
\draw[->] (P_0) to node {$\alpha_0$} (Pv_0);
\end{tikzpicture}
 \end{center}
such that all the squares commute.

We will use the notation and result of \cite[Section 4]{Bock}, working over $\mathbb{C}$. For $T$ a finite dimensional $S^e$-module let $F_T$ be the $\A^e$-module $\A \otimes_S T \otimes_S \A$, and let $T^*$ denote the $\mathbb{C}$ dual of $T$. 
Then
\begin{eqnarray*}
F_{T^*} & \rightarrow & F_T^{\vee} \\
(1 \otimes \phi \otimes 1) & \mapsto & \left( (1 \otimes p \otimes 1) \mapsto \sum_{e,f \in Q_0} \phi(e p f ) e \otimes_{\mathbb{C}} f  \right)
\end{eqnarray*}
gives an isomorphism of $\A^e$-modules. Moreover, as $S$ is semisimple, tensoring ${\A \otimes_S (-) \otimes_S \A}$ is flat, hence constructing an isomorphism of $S^e$-modules $\mathcal{W}_j \rightarrow \mathcal{W}_{n-j}^*$ will give us an isomorphism of $\A^e$-modules $F_{\mathcal{W}_j} \rightarrow F_{\mathcal{W}_{n-j}^*}$, and composing with the above isomorphism an will give us an isomorphism $F_{\mathcal{W}_j} \rightarrow  F_{\mathcal{W}_{n-j}}^{\vee}$ .

For $|p|=n-j$ define $\partial_p \in \mathcal{W}^*_{n-j}$ by $\partial_p (\delta_q \Phi_n)$ = $c_{qp}$, where $\Phi_n= \sum_{|t|=n}c_t t$ and $|q|=j$. Then there are $S^e$-module homomorphisms
\begin{eqnarray*}
\eta_j: \mathcal{W}_j & \rightarrow & \mathcal{W}_{n-j}^*\\
\delta_p \Phi_n & \mapsto & \gamma_j\partial_p 
\end{eqnarray*}
for $\gamma_j$ arbitrary nonzero constants.\\

Note that $\partial_p (\delta_q \Phi_n) = c_{qp}=(-1)^{(n-1)|q|} c_{pq}=(-1)^{(n-1)|q|} \partial_q (\delta_p \Phi_n)$. To see $\eta_j$ is injective suppose $\eta_j (\sum \lambda_p \delta_p \Phi_n)=0$. Then $\sum_p \lambda_p \partial_p (\delta_q \Phi_n) = \sum_p \lambda_p c_{qp} = 0$ for all $q$. Hence $\sum_p \lambda_p \delta_p \Phi_n = \sum_{p,q} \lambda_p c_{pq} q = \sum_q (\sum_p \lambda_p c_{pq})q = 0$, so the map is injective.

There are $\mathbb{C}$ vector space isomorphisms between $e \mathcal{W}_j f $ and $(f \mathcal{W}_{n-j}e)^*$ for $e,f \in Q_0$ arising from the pairings,
\begin{eqnarray*}
 e \mathcal{W}_j f \otimes_{\mathbb{C}} f \mathcal{W}_{n-j} e & \rightarrow & \mathbb{C} \\
e \delta_p \Phi_n f  \otimes f \delta_q \Phi_n e & \mapsto &  \gamma_j c_{qp}=\eta_j(\delta_p \Phi_n)(\delta_q \Phi_n).
\end{eqnarray*}
As $ e \mathcal{W}_j f$ and $f \mathcal{W}_{n-j} e$ are finite dimensional $\mathbb{C}$ vector spaces the injectivity of $\eta_j$ implies this pairing is perfect, thus $\eta_j$ is an isomorphism.

Now we use this to define isomorphisms $\alpha_j$

\begin{eqnarray*}
\alpha_j: \, \A \otimes_S \mathcal{W}_j \otimes_S \A & \rightarrow & (\A \otimes_S \mathcal{W}_{n-j} \otimes_S \A)^{\vee} \\
(a_1 \otimes \delta_p \Phi_n \otimes a_2 ) &\mapsto& \big( (b_1 \otimes \delta_q \Phi_n \otimes b_2) \mapsto (a_2 b_1 \otimes  \eta_j(\delta_p \Phi_n) \delta_q \Phi_n \otimes b_2 a_1) \big) \\
&= & \big(  (b_1 \otimes \delta_q \Phi_n \otimes b_2) \mapsto ( \gamma_j c_{qp} h(p)  a_2 b_1 t(q) \otimes_{\mathbb{C}} h(q) b_2 a_1t(p) )\big)
\end{eqnarray*}

It remains to check $\alpha_{j-1} \circ d_j = -d^{\vee}_{n-j+1} \circ \alpha_j$, i.e. the following diagram commutes;

\begin{center}
\begin{tikzpicture}[auto]
\node (P_n) at (0,0) {$\A \otimes_S \mathcal{W}_j \otimes_S \A$};
\node (P_{n-1}) at (5,0) {$\A \otimes_S \mathcal{W}_{j-1} \otimes_S \A$};

\node (Pv_n) at (0,-3) {$(\A \otimes_S \mathcal{W}_{n-j} \otimes_S \A)^{\vee}$};
\node (Pv_{n-1}) at (5,-3) {$(\A \otimes_S \mathcal{W}_{n-j+1} \otimes_S \A)^{\vee}$};

\draw[->] (P_n) to node {$d_j$} (P_{n-1});

\draw[->] (Pv_n) to node {$-d_{n-j+1}^{\vee}$} (Pv_{n-1});

\draw[->] (P_n) to node {$\alpha_j$} (Pv_n);
\draw[->] (P_{n-1}) to node {$\alpha_{j-1}$} (Pv_{n-1});

\end{tikzpicture}
 \end{center}

If $|p|=n-j$ and $|q|=j-1$ then
\begin{eqnarray*}
& \hfil & ( \alpha_{j-1} \circ d_j(1 \otimes \delta_p \Phi \otimes 1))(1 \otimes \delta_q \Phi \otimes 1) \\
&=&  \alpha_{j-1}(\epsilon_j \sum_a a \otimes \delta_{pa}\Phi \otimes 1 + (-1)^{j+(n-1)}  1 \otimes \delta_{ap} \Phi \otimes a)(1 \otimes \delta_q \Phi \otimes 1) \\
&=& \gamma_{j-1} \epsilon_{j} \sum_a c_{qpa} t(q) \otimes a +\gamma_{j-1} \epsilon_j (-1)^{j+n-1} \sum_a c_{qap} a \otimes h(q) \end{eqnarray*}
whilst
\begin{eqnarray*}
& \hfil & (-d_{n-j+1}^{\vee} \circ \alpha_j(1 \otimes \delta_p \Phi \otimes 1))(1 \otimes \delta_q \Phi \otimes 1)  \\
&=& \alpha_{j}(1 \otimes \delta_p \Phi \otimes 1)(-\epsilon_{n-j+1} \sum_a a \otimes \delta_{qa} \Phi \otimes 1 +(-1)^{j} 1 \otimes \delta_{aq} \Phi \otimes a) \\
&=& -\gamma_j \epsilon_{n-j+1} \sum_a c_{qap} a \otimes t(p) -\gamma_j\epsilon_{n-j+1} (-1)^{j} \sum_a c_{aqp} h(p) \otimes a
\end{eqnarray*}
where the sums are taken over all $a \in Q_1$. Thus for these to be equal we require
\begin{equation*}
-\gamma_j \epsilon_{n-j+1} c_{qap} = \gamma_{j-1}\epsilon_j (-1)^{j+n-1}c_{qap}
\end{equation*}
and
\begin{equation*}
-\gamma_j \epsilon_{n-j+1} (-1)^j c_{aqp} = \gamma_{j-1} \epsilon_j c_{qpa}=(-1)^{n-1}\gamma_{j-1} \epsilon_j c_{aqp} 
\end{equation*}
for all $a \in Q_1$. These follow if $(-1)^n \gamma_{j-1} \epsilon_j= \gamma_j \epsilon_{n-j+1}(-1)^j$ for $j=1, \dots , n$. As the $\gamma_j$ were arbitrary non-zero scalars, we can choose the $\gamma_j$ so that this is satisfied, and the proof is completed.
\end{proof}

Let $Q$ be a single vertex quiver and $A=\mathcal{D}(\Phi_n,n-2)$ be a Noetherian, Koszul, $n$-CY algebra. We consider a PBW deformation, $\A$, given by $\theta_1,\theta_0$, and set $\phi_{n-1}=  \sum_p p\theta_1(\delta_p \Phi_n)=\sum c_q q$.  If the deformation is a zeroPBW deformation $\phi_{n-1}$ has the $n$-superpotential property. 

\begin{Extra}\label{Extra}
Keeping the above notation and assumptions
\begin{itemize}
\item[1.] Any zeroPBW deformation of $A$ is $n$-CY
\item[2.] Let $n=2$ or $3$. Then the zeroPBW deformations of $A$ are exactly the $n$-CY PBW deformations of $A$, and moreover any superpotential $\Phi'=\Phi_n+\phi_{n-1}+\phi_{n-2}$ is $(n-2)$-coherent.

\end{itemize}
\end{Extra}

\begin{proof}
1. Let $\mathcal{A}$ be a PBW deformation of $A$, defined by a map $\theta$. Then define $\Phi'(\theta)=\Phi_n+\phi_{n-1}+\phi_{n-2}$ by 
\begin{equation*}
\phi_{n-2+j}:= -\sum_{|p|=n-2} p \theta_{j}(\delta_p \Phi_n).
\end{equation*}
 Write $\phi_{n-1}=\sum_t c_t t$. We will show that $\mathcal{A}$ is $n$-CY if and only if 
\begin{eqnarray*}
0&=& \sum_{i=0}^{n-2} (-1)^{in} c_{q_{i+2} \dots q_{n-1} q_1 \dots q_{i+1} }  
\end{eqnarray*}
for any $q=q_1 \dots q_{n-1}$, with $q_j \in Q_1$. In particular this shows any zeroPBW deformation is $n$-CY, as then $\Phi'(\theta)$ is a superpotential and when $n$ is even $\phi_{n-1}=\sum c_t t = 0$ and when $n$ is odd the terms cancel in pairs by the superpotential property.

Referring to Theorem \ref{WZ} $\A$ is $n$-CY if and only if 
\begin{eqnarray*}
0&=&\sum^{n-2}_{i=0} (-1)^i id^{\otimes i} \otimes \theta_1 \otimes id^{\otimes(n-2-i)}( \Phi_n) \\
&=&\sum_i (-1)^i \sum_{p=p_1 \dots p_{n-2}} p_1 \dots p_i \theta_1(\delta_{p_1 \dots p_i} \Phi_n\delta'_{p_{i+1} \dots p_{n-2}})p_{i+1} \dots p_{n-2}  \\
&=&\sum_i (-1)^{i+(n-1)(n-2-i)} \sum_p p_1 \dots p_i \theta_1(\delta_{p_{i+1} \dots p_{n-2} p_1 \dots p_i}\Phi_n)p_{i+1} \dots p_{n-2}  \\
&\hfil&\big( \text{using $\theta_1 (\delta_q \Phi_n)=\delta_q \phi_{n-1}$  and writing $\phi_{n-1}= \sum_{|r|=n-1} c_r r$} \big) \\
&=&\sum_i (-1)^{in} \sum_{p,a} p_1 \dots p_i c_{p_{i+1} \dots p_{n-2} p_1 \dots p_i a}a p_{i+1} \dots p_{n-2}  \\
&=&\sum_i (-1)^{in} \sum_{q} c_{q_{i+2} \dots q_{n-1} q_1 \dots q_i q_{i+1}} q_1 \dots q_i  q_{i+1} q_{i+2} \dots q_{n-1}  \\
\end{eqnarray*}

Considering the coefficients of the paths, which are linearly independent, and calculating the coefficient of $q_1 \dots q_{n-1}$ we find the condition on $\phi_{n-1}$ to be

\begin{eqnarray*}
0&=& \sum_{i=0}^{n-2} (-1)^{in} c_{q_{i+2} \dots q_{n-1}  q_1 \dots q_{i+1} }  
\end{eqnarray*}
for all $q_1 \dots q_{n-1}$. 
\\
2. By part 1/Theorem \ref{WZ} we have a condition for $\A$ to be $n$-CY.  In the $n=3$ case the condition gives that $\Im id \otimes \theta_1 - \theta_1 \otimes id=\psi(\theta_1)= \{ 0 \}$ which is the zeroPBW condition. When $n=2$ it gives the condition $\theta_1 = 0$ which is the zeroPBW condition for even $n$.

Moreover when $n=3$ any superpotential, $\Phi'$, is $1$-coherent. In particular the fact that $A$ is 3-CY gives a duality in $\mathcal{W^{\bullet}}$ between the 1st and 2nd terms - the arrows $a \in Q_1$ and the relations $\delta_{a}\Phi_3$. Hence as the arrows are linearly independent so are the relations, and $\sum_{a \in Q_1} \lambda_a \delta_a \Phi_3 = 0 \Rightarrow \lambda_a = 0 \Rightarrow \sum \lambda_a \delta_a \Phi' = 0$, so $\Phi'$ is $1$-coherent. When $n=2$ the superpotential is only differentiated by paths of length 0, so it is clearly $0$-coherent.
\end{proof}

This result in the $3$-CY case is is already known in the context of a general quiver due to the following result of Berger and Taillefer.

\begin{TheoremBT}[{\cite[Section 3]{BT}}] \label{TheoremBT}
Let $A=\mathcal{D}(\Phi_{N+1},1)$ be $N$ Koszul and $3$-CY, with $\Phi_{N+1}$ a superpotential on some quiver $Q$. Then a PBW deformation $\mathcal{A}$ of A is $3$-CY if and only if it is a zeroPBW deformation. Moreover the zeroPBW deformations correspond to $\mathcal{D}(\Phi',1)$ for $\Phi'=\Phi_n+\phi_{n-1}+\dots + \phi_1$ an inhomogeneous superpotential.
\end{TheoremBT}

\section{Application: Symplectic reflection algebras}
In this section we recall the definition of symplectic reflection algebras and deduce they are Morita equivalent to CY superpotential algebras by applying Theorems \ref{BSW3}, \ref{Theorem3}, and \ref{TheoremCY}. We go on to calculate some examples, and consider the interpretation of the parameters of a symplectic reflection algebra in the superpotential setting.

\vspace{0.4cm}

Let $V$ be a $2n$ dimensional space, equipped with symplectic form $\omega$ and $G$ a symplectic reflection group which acts faithfully on $V$ preserving the symplectic form. We say $(G,V)$ is indecomposable if there is no $G$-stable splitting $V=V_1 \oplus V_2$ with $\omega(V_1,V_2)=0$. We consider the skew group algebra $\mathbb{C}[V] \rtimes G$ and the symplectic reflection algebras are defined to be the PBW deformations of this relative to $\mathbb{C} G$, and were classified by Etingof and Ginzburg \cite{EG}. 

\begin{EG}( \cite[Theorem 1.3]{EG}) \label{EG}
Any PBW deformation of such an indecomposable $\mathbb{C}[V] \rtimes G$ is of the form 
\begin{equation*}
H_{t,\mathbf{c}}=\frac{T_{\mathbb{C}G} (V^*)}{<[x,y]-\kappa_{t,\mathbf{c}}(x,y) : x,y \in V^*>}
\end{equation*}
where $\kappa_{t,\mathbf{c}}(x,y)=t\omega_{V^*}(x,y)-\Sigma_s \omega_s(x,y) \mathbf{c}(s) s$ with the sum taken over the symplectic reflections $s$, and $\mathbf{c}$ a complex valued class function on symplectic reflections. The symplectic form $\omega_{V^*}$ on $V^*$ is induced from $\omega$ on $V$, and $\omega_s$ is defined as $\omega_{V^*}$ restricted to $(id-s)(V^*)$.
\end{EG} In particular they fall into the even dimensional case considered in Section \ref{Main Results} above, with the PBW parameter $\theta_1$ equal to zero, and they are PBW deformations relative to $\mathbb{C}G$, which under Morita equivalence with a quiver of relations correspond to PBW deformations relative to $S$.

Two infinite classes of symplectic reflection algebras are the rational Cherednik and wreath product algebras.
\begin{itemize}
\item \textbf{Rational Cherednik Algebras.} Let $\mathfrak{h}$ be a finite dimensional vector space and $G$ a finite subgroup of $\GL(\mathfrak{h})$. Then $V=\mathfrak{h}\oplus \mathfrak{h}^{*}$ has the natural symplectic form $\omega((x,f),(y,g))=f(y)-g(x)$, and an action of $G$ as a symplectic reflection group. Then the symplectic reflection algebras given by PBW deformations of $\mathbb{C}[V] \rtimes G$ are the Rational Cherednik Algebras.
\item \textbf{Wreath Product Algebras.} Let $K$ be a finite subgroup of $\SL_2(\mathbb{C})$, and $S_n$ the symmetric group of order $n$. Then the wreath product group $G=S_n \wr K$ is a symplectic reflection group acting on $V=(\mathbb{C}^2)^n$.
\end{itemize}

\subsection{Symplectic reflection algebras as superpotential algebras}
Here we show that symplectic reflection algebras are Morita equivalent to superpotential algebras.

\begin{Sym} \label{Sym}
Let $H_{t,\mathbf{c}}$ be as in Theorem \ref{EG}. 
\begin{itemize}
\item[1.] $H_{0,0}$ is Morita equivalent to $A=\mathcal{D}(\Phi_{2n},2n-2)$  for some homogeneous superpotential $\Phi_{2n}$, and is 2n-CY and Koszul.
\item[2.] Any $H_{t,\mathbf{c}}$ is Morita equivalent to $\A=\mathcal{D}(\Phi',2n-2)$ for $\Phi'=\Phi_{2n} + \phi_{2n-2}$ an inhomogeneous superpotential which is $(2n-2)$-coherent.
\item[3.] Any $(2n-2)$-coherent superpotential of the form $\Phi'=\Phi_{2n}+\phi_{2n-2}$, gives an algebra $\A=D(\Phi',2n-2)$ that is Morita equivalent to a symplectic reflection algebra $H_{t,\mathbf{c}}$.
\item[4.] All $H_{t,\mathbf{c}}$ are $2n$-CY algebras.
\end{itemize}
\end{Sym}
\begin{proof} 

1: $H_{0,0}$ is just $\mathbb{C}[V] \rtimes G$. Hence Theorem \ref{BSW3} applies, and $H_{0,0}$ is Morita equivalent to a $2n$-CY, Koszul algebra $A:=\mathcal{D}(\Phi_{2n},2n-2)$ for the McKay quiver for $(G,V)$.

2 and 3:  The Morita equivalence between $\mathbb{C}[V] \rtimes G$ and $A$ switches $\mathbb{C}G$ with $S$, respects the gradings, and respects the Koszul resolutions. Hence any zeroPBW deformation of $\mathbb{C}[V] \rtimes G$ corresponds to a zeroPBW deformation of $A$, and 2 and 3 follow from Theorem \ref{Theorem3} once we note by Theorem \ref{EG} all PBW deformations of $\mathbb{C}[V]\rtimes G$  are zeroPBW.

4: From 2 and 3 we know any symplectic reflection algebra is Morita equivalent to some $\A:=\mathcal{D}(\Phi_{2n} + \phi_{2n-2},2n-2)$. By Theorem \ref{TheoremCY} $\A$ is $2n$-CY: $\A$ has a finite length resolution by finitely generated projective modules and $\RHom(\A,\A^e)[n] \cong \A$. Then the Morita equivalent symplectic reflection algebra, $H_{t,\mathbf{c}}$, also has a finite length resolution by finitely generated projective modules, as these properties are preserved under Morita equivalence. The isomorphism in the derived category $\RHom(\A,\A^e)[n] \cong \A$ transfers to the isomorphism $\RHom(H_{t,\mathbf{c}},H_{t,\mathbf{c}}^e)[n] \cong H_{t,\mathbf{c}}$ under the Morita equivalence, hence $H_{t,\mathbf{c}}$ is $2n$-CY.
\end{proof} 

\subsubsection{Examples} We give three examples.
\begin{example4} \label{example4}
Here we consider the symplectic reflection algebra corresponding to the group $S_3$ acting on $\mathfrak{h} \oplus \mathfrak{h}^*$ in the manner of a rational Cherednik algebra, where the representation $\mathfrak{h}$ is given by
\begin{equation*}
S_3=\left\langle  g=\left( \begin{array}{c c}\varepsilon_3 &0\\ 0&\varepsilon_3^2 \end{array} \right), \, h=\left( \begin{array}{c c} 0&1\\ 1&0 \end{array} \right) \right\rangle
\end{equation*}
with $\varepsilon_3$ a primitive third root of unity.

The McKay quiver is;
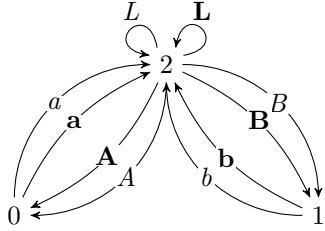
\begin{center}
$
\begin{tikzpicture} [bend angle=45, looseness=1]
\node (C1) at (0,0)  {$0$};
\node (C2) at (2,2)  {$2$};
\node (C2a) at (2.1,2.1) {};
\node (C2b) at (1.9,2.1) {};
\node (C3) at (4, 0){$1$};
\draw [->,bend left] (C1) to node[gap]  {\emph{a}} (C2);
\draw [->,bend left=20,looseness=1] (C1) to node[gap]{\textbf{a}} (C2);
\draw [->,bend left] (C2) to node[gap]{\emph{A}} (C1);
\draw [->,bend left=20,looseness=1] (C2) to node[gap,below=-5pt]{\textbf{A}} (C1);
\draw [->,bend left] (C3) to node[gap]  {\emph{b}} (C2);
\draw [->,bend left=20,looseness=1] (C3) to node[gap]{\textbf{b}} (C2);
\draw [->,bend left] (C2) to node[gap]{\emph{B}} (C3);
\draw [->,bend left=20,looseness=1] (C2) to node[gap,below=-5pt]{\textbf{B}} (C3);
\draw[->]  (C2b) edge [in=170,out=100,loop,looseness=12] node[above] {\emph{L}} (C2b);
\draw[->]  (C2a) edge [in=80,out=10,loop,looseness=12] node[above] {\textbf{L}} (C2a);
\end{tikzpicture}
$
\end{center}

Now following the calculation in BSW \cite[Theorem 3.2]{BSW} we can calculate a superpotential to accompany this quiver, $\Phi_4$
\begin{eqnarray*}
\Phi_4=
 &-\text{\emph{A}\textbf{a}\emph{A}\textbf{a}}  
+2\text{\textbf{A}\emph{a}\emph{A}\textbf{a}}  
+4\text{\emph{A}\textbf{L}\emph{L}\textbf{a}}  
-4\text{\textbf{A}\emph{L}\emph{L}\textbf{a}}
-4\text{\emph{A}\emph{b}\textbf{B}\textbf{a}} 
+2\text{\emph{A}\textbf{b}\emph{B}\textbf{a}}
+2\text{\textbf{A}\emph{b}\emph{B}\textbf{a}} 
\\
 &-\text{\textbf{A}\emph{a}\textbf{A}\emph{a}} 
-4\text{\emph{A}\textbf{L}\textbf{L}\emph{a}} 
+4\text{\textbf{A}\emph{L}\textbf{L}\emph{a}} 
+2\text{\emph{A}\textbf{b}\textbf{B}\emph{a}} 
+2\text{\textbf{A}\emph{b}\textbf{B}\emph{a}} 
-4\text{\textbf{A}\textbf{b}\emph{B}\emph{a}} 
-8\text{\emph{L}\emph{L}\textbf{L}\textbf{L}} 
\\
& +4\text{\emph{b}\emph{B}\textbf{L}\textbf{L}} 
-4\text{\emph{b}\textbf{B}\emph{L}\textbf{L}}
+4\text{\emph{L}\textbf{b}\emph{B}\textbf{L}}
+4\text{\textbf{b}\textbf{B}\emph{L}\emph{L}} 
-2\text{\emph{b}\emph{B}\textbf{b}\textbf{B}}
+\text{\emph{b}\textbf{B}\emph{b}\textbf{B}}
+\text{\textbf{b}\emph{B}\textbf{b}\emph{B}} 
\\
& +\text{cyclic permutations}
\end{eqnarray*}

We now wish to look at the zeroPBW deformations, which by Theorem \ref{Sym} correspond to $2$-coherent superpotentials $\Phi'=\Phi_{4}+\phi_2$. Writing $\phi_2= \sum c_{xy} xy$ the PBW deformations are parametrised by the $c_{xy}$ such that $\Phi'$ is a $2$-coherent superpotential. We see that $\Phi'$ is a superpotential if $c_{xy}=-c_{yx}$ for all arrows $x,y$. A superpotential $\Phi'$ is $2$-coherent if the $\delta_{xy}\Phi'= \delta_{xy}\Phi_4+c_{xy}e_{h(x)}e_{t(y)}$  satisfy the same linear relations as the $\delta_{xy} \Phi_4$.  For instance as $\delta_{\emph{Aa}}\Phi=0$ and  $\delta_{\emph{A}\textbf{a}}\Phi = - \delta_{\textbf{A}\emph{a}}$ we require that $c_{\emph{Aa}}=0$ and $c_{\emph{A}\textbf{a}}=-c_{\textbf{A}\emph{a}}$.

Making these calculations in this example we see the only non-zero $c$, and dependency relations among them, are:
\begin{align*}
c_{\textbf{a}\emph{A}}=-c_{\emph{A}\textbf{a}}=c_{\textbf{A}\emph{a}}=-c_{\emph{a}\textbf{A}} \\
c_{\textbf{b}\emph{B}}=-c_{\emph{B}\textbf{b}}=c_{\textbf{B}\emph{b}}=-c_{\emph{b}\textbf{B}} \\
c_{\textbf{a}\emph{A}}+c_{\textbf{b}\emph{B}}=c_{\textbf{L}\emph{L}}=-c_{\emph{L}\textbf{L}}
\end{align*}
Hence we have a 2-coherent superpotential for any
\begin{equation*}
\phi_2=c_{\textbf{a}\emph{A}}(\textbf{a}\emph{A}-\emph{A}\textbf{a}+\textbf{A}\emph{a}-\emph{a}\textbf{A})+c_{\textbf{b}\emph{B}}(\textbf{b}\emph{B}-\emph{B}\textbf{b}+\textbf{B}\emph{b}-\emph{b}\textbf{B})+(c_{\textbf{a}\emph{A}}+c_{\textbf{b}\emph{B}})(\textbf{L}\emph{L}-\emph{L}\textbf{L})
\end{equation*}

So we have 2 degrees of freedom in our parameters, exactly as we do for the $t,\mathbf{c}$ in $H_{t,\mathbf{c}}$ for $S_3$ acting on $\mathfrak{h} \oplus\mathfrak{h}^*$. \\
\end{example4}

\begin{example5} \label{example5}
Here we consider the symplectic reflection algebra corresponding to the dihedral group of order 8, $D_8$, acting on $\mathfrak{h} \oplus \mathfrak{h}^*$ in the manner of a rational Cherednik algebra. The representation $\mathfrak{h}$ is given as
\begin{equation*}
D_8=\left\langle\sigma=\left(\begin{array}{c c}\varepsilon_4 &0\\ 0&\varepsilon_4^3 \end{array} \right), \, \tau=\left(\begin{array}{c c}0&1\\ 1&0 \end{array} \right) \right\rangle
\end{equation*}
where $\varepsilon_4$ is a primitive fourth root of unity.

This has McKay quiver
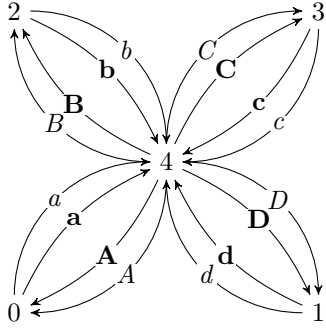
\begin{center}
$
\begin{tikzpicture} [bend angle=45, looseness=1]
\node (C1) at (0,0)  {$0$};
\node (C2) at (2,2)  {$4$};
\node (C3) at (4, 0){$1$};
\node (C4) at (0,4) {$2$};
\node (C5) at (4,4) {$3$};
\draw [->,bend left] (C1) to node[gap]  {\emph{a}} (C2);
\draw [->,bend left=20,looseness=1] (C1) to node[gap] {\textbf{a}} (C2);
\draw [->,bend left] (C2) to node[gap] {\emph{A}} (C1);
\draw [->,bend left=20,looseness=1] (C2) to node[gap] {\textbf{A}} (C1);
\draw [->,bend left] (C3) to node[gap]  {\emph{d}} (C2);
\draw [->,bend left=20,looseness=1] (C3) to node[gap] {\textbf{d}} (C2);
\draw [->,bend left] (C2) to node[gap] {\emph{D}} (C3);
\draw [->,bend left=20,looseness=1] (C2) to node[gap] {\textbf{D}} (C3);
\draw [->,bend left] (C4) to node[gap]  {\emph{b}} (C2);
\draw [->,bend left=20,looseness=1] (C4) to node[gap] {\textbf{b}} (C2);
\draw [->,bend left] (C2) to node[gap] {\emph{B}} (C4);
\draw [->,bend left=20,looseness=1] (C2) to node[gap] {\textbf{B}} (C4);
\draw [->,bend left] (C5) to node[gap]  {\emph{c}} (C2);
\draw [->,bend left=20,looseness=1] (C5) to node[gap] {\textbf{c}} (C2);
\draw [->,bend left] (C2) to node[gap] {\emph{C}} (C5);
\draw [->,bend left=20,looseness=1] (C2) to node[gap] {\textbf{C}} (C5);
\end{tikzpicture}
$
\end{center}

By choosing a $G$-equivariant basis we calculate the superpotential
\begin{eqnarray*}
\Phi=
 &-\text{\emph{A}\textbf{a}\emph{A}\textbf{a}}  
+2\text{\textbf{A}\emph{a}\emph{A}\textbf{a}}  
-4\text{\emph{A}\emph{d}\textbf{D}\textbf{a}}  
+2\text{\emph{A}\textbf{d}\emph{D}\textbf{a}} 
+2\text{\textbf{A}\emph{d}\emph{D}\textbf{a}}
+2\text{\emph{A}\textbf{b}\emph{B}\textbf{a}}
-2\text{\textbf{A}\emph{b}\emph{B}\textbf{a}}
+2\text{\emph{A}\textbf{c}\emph{C}\textbf{a}}
-2\text{\textbf{A}\emph{c}\emph{C}\textbf{a}}
\\
 &-\text{\textbf{A}\emph{a}\textbf{A}\emph{a}}  
-4\text{\textbf{A}\textbf{d}\emph{D}\emph{a}}  
+2\text{\textbf{A}\emph{d}\textbf{D}\emph{a}} 
+2\text{\emph{A}\textbf{d}\textbf{D}\emph{a}}
+2\text{\textbf{A}\emph{b}\textbf{B}\emph{a}}
-2\text{\emph{A}\textbf{b}\textbf{B}\emph{a}}
+2\text{\textbf{A}\emph{c}\textbf{C}\emph{a}}
-2\text{\emph{A}\textbf{c}\textbf{C}\emph{a}} 
-\text{\emph{D}\textbf{d}\emph{D}\textbf{d}}  
\\
 &
+2\text{\textbf{D}\emph{d}\emph{D}\textbf{d}}  
+2\text{\emph{D}\textbf{b}\emph{B}\textbf{d}}
-2\text{\textbf{D}\emph{b}\emph{B}\textbf{d}}
+2\text{\emph{D}\textbf{c}\emph{C}\textbf{d}}
-2\text{\textbf{D}\emph{c}\emph{C}\textbf{d}}
-\text{\textbf{D}\emph{d}\textbf{D}\emph{d}}  
+2\text{\textbf{D}\emph{b}\textbf{B}\emph{d}}
-2\text{\emph{D}\textbf{b}\textbf{B}\emph{d}}
+2\text{\textbf{D}\emph{c}\textbf{C}\emph{d}}
\\
 &
-2\text{\emph{D}\textbf{c}\textbf{C}\emph{d}}
-\text{\emph{B}\textbf{b}\emph{B}\textbf{b}}  
+2\text{\textbf{B}\emph{b}\emph{B}\textbf{b}}
-4\text{\emph{B}\emph{c}\textbf{C}\textbf{b}}
+2\text{\emph{B}\textbf{c}\emph{C}\textbf{b}}
+2\text{\textbf{B}\emph{c}\emph{C}\textbf{b}}
-4\text{\textbf{B}\textbf{c}\emph{C}\emph{b}}
-\text{\textbf{B}\emph{b}\textbf{B}\emph{b}}  
+2\text{\textbf{B}\emph{c}\textbf{C}\emph{b}}
\\
 &
+2\text{\emph{B}\textbf{c}\textbf{C}\emph{b}}
-\text{\emph{C}\textbf{c}\emph{C}\textbf{c}}  
+2\text{\textbf{C}\emph{c}\emph{C}\textbf{c}}
-\text{\textbf{C}\emph{c}\textbf{C}\emph{c}} 
\\
& + \text{cyclic permutations}
\end{eqnarray*}

We now calculate the zeroPBW deformations of $A:=\mathcal{D}(\Phi_4,2)$. We write $\phi_2=\sum c_{xy} xy$, and by Theorem \ref{Sym} the zeroPBW deformations of $A$  are parameterised by the $c_{xy}$ such that $\Phi'=\Phi_4+\phi_2$ is a 2-coherent superpotential.

In particular we require $c_{xy}=-c_{yx}$ and for the $c_{xy}$ to satisfy the same linear relations as the $\delta_{xy} \Phi_4$. Making these calculations in this example we see the only non-zero $c_{xy}$, and dependency relations among them, are;
\begin{eqnarray*}
c_{\textbf{a}\emph{A}}=-c_{\emph{A}\textbf{a}}=c_{\textbf{A}\emph{a}}=-c_{\emph{a}\textbf{A}} \\
c_{\textbf{b}\emph{B}}=-c_{\emph{B}\textbf{b}}=c_{\textbf{B}\emph{b}}=-c_{\emph{b}\textbf{B}} \\
c_{\textbf{c}\emph{C}}=-c_{\emph{C}\textbf{c}}=c_{\textbf{C}\emph{c}}=-c_{\emph{c}\textbf{C}} \\
c_{\textbf{d}\emph{D}}=-c_{\emph{D}\textbf{d}}=c_{\textbf{D}\emph{d}}=-c_{\emph{d}\textbf{D}} \\
c_{\textbf{a}\emph{A}}+c_{\textbf{d}\emph{D}}=-c_{\emph{b}\textbf{B}}-c_{\emph{c}\textbf{C}}=c_{\textbf{b}\emph{B}}+c_{\textbf{c}\emph{C}}
\end{eqnarray*}
Hence to obtain a 2-coherent superpotential we require a $\phi_2$ of the form
\begin{eqnarray*}
\phi_2=&c_{\textbf{a}\emph{A}}(\textbf{a}\emph{A}-\emph{A}\textbf{a}+\textbf{A}\emph{a}-\emph{a}\textbf{A})+c_{\textbf{b}\emph{B}}(\textbf{b}\emph{B}-\emph{B}\textbf{b}+\textbf{B}\emph{b}-\emph{b}\textbf{B})+c_{\textbf{c}\emph{C}}(\textbf{c}\emph{C}-\emph{C}\textbf{c}+\textbf{C}\emph{c}-\emph{c}\textbf{C})\\
&+(c_{\textbf{b}\emph{B}}+c_{\textbf{c}\emph{C}}-c_{\textbf{a}\emph{A}})(\textbf{d}\emph{D}-\emph{D}\textbf{d}+\textbf{D}\emph{d}-\emph{d}\textbf{D})
\end{eqnarray*}

So we see here there are 3 degrees of freedom in our parameters exactly as for the parameters $t$ and $\mathbf{c}$ in the symplectic reflection algebra for $D_8$ acting on $\mathfrak{h} \oplus \mathfrak{h}^*$.\\
\end{example5}

\begin{examplepp}
A special case of path algebras Morita equivalent to symplectic reflection algebras are the deformed preprojective algebras of \cite{CBH}. These can be given as superpotential algebras in the $n=2$, differentiation by paths of length 0, case.\\

We consider a skew group algebras $\mathbb{C}[\mathbb{C}^2] \rtimes G$ for $G$ is a finite subgroup of $\SL_2(\mathbb{C})$. We construct the McKay quiver and label the arrows in a particular way; between any two vertices we choose a direction, label the arrows in this direction $a_1,..., a_k$, and the arrows in the opposite direction $a_1^*, \dots a_k^*$. Then $\mathbb{C}[\mathbb{C}^2] \rtimes G$ is Morita equivalent to $A=\mathcal{D}(\Phi_2,0)$  for the  homogeneous superpotential, $\Phi_2=\sum [a,a^*]$. This is the preprojective algebra
\begin{equation*}
A=\mathcal{D}(\Phi_2,0)=\frac{\mathbb{C}Q}{\sum [a,a^{*}]}
\end{equation*} 

 Now we consider the PBW deformations of $A$. We apply Theorem \ref{Sym}, and deduce PBW deformations correspond to $0$-coherent inhomogeneous superpotentials, $\Phi_2+\phi_0$. We note that $\phi_0:=-\sum_{i \in Q_0} \lambda_i e_i \in S$ can in fact be arbitrary as any element of $S$ satisfies the superpotential property, and the $0$-coherent property is always satisfied. Hence we recover that PBW deformations of the preprojective algebra are the deformed preprojective algebras 
\begin{equation*}
\A=\mathcal{D}(\Phi_2+\phi_0,0)=\frac{\mathbb{C}Q}{\sum [a,a^{*}] -\sum \lambda_i e_i},
\end{equation*} which are parameterised by a scalar, $\lambda_i$, for each vertex, $i$. By Theorem \ref{Sym} these are 2-CY.

\end{examplepp}

\section{Application: PBW deformations of skew group rings for $\GL_2$} \label{General}
So far we have been considering subgroups of $\SL(W)$ where $W$ is a finite dimensional vector space. This corresponds to non-twisted superpotentials. Here we consider algebras Morita equivalent to $\mathbb{C}[W] \rtimes G$ for $G$ a finite subgroup of $\GL(W)$, and  the existence of PBW deformations for $\mathbb{C}[W] \rtimes G$. 

We recall 
\begin{Gen}[{\cite[3.2, 6.1 and 6.8]{BSW}}]
Let $W=\mathbb{C}^n$, and $G$ be a finite subgroup of $\GL(W)$. Then $\mathbb{C}[W] \rtimes G$ is Morita equivalent to $\mathcal{D}(\Phi_n,n-2)$ for the McKay quiver $(G,W)$, $\Phi_n$ a twisted homogeneous superpotential, and the twist automorphism given by $(-)\otimes_{\mathbb{C}} \det W$. There is a recipe to construct the twisted superpotential.
\end{Gen}

Working with homogeneous superpotentials, $\Phi_n=\sum c_p p $, we have $c_p=0$ for any $p$ that is not a closed path. This is no longer the case for twisted homogeneous superpotentials, here we find $c_p=0$ unless $h(p)=\sigma(t(p))$. Since the twist for $\mathbb{C}[W] \rtimes G$ is given by tensoring by $\det W$, $c_p$ is non zero only for paths from $W_i$ to $\det W \otimes_{\mathbb{C}} W_i$, where the $W_i$ are the irreducible representations corresponding to vertices in the McKay quiver.

There are two different cases of finite subgroups of $\GL_n(\mathbb{C})$ we consider, those that contain pseudo-reflections, and those that do not. Those that do not are known as \emph{small} subgroups.

As a particular case we will consider $\GL_2(\mathbb{C})$, where differentiation is by paths of length $0$, and so our relations are given by the superpotential, and any relations are a sum of paths with tail $W_i$ and head $\det W \otimes_{\mathbb{C}} W_i $.

\begin{glcase}\label{glcase}
Let $G$ be a small finite subgroup of $\GL_2(\mathbb{C})$, which is not contained in $\SL_2(\mathbb{C})$. Then $\mathbb{C}[\mathbb{C}^2] \rtimes G$ has no nontrivial (relative to $\mathbb{C}G$) PBW deformations.
\end{glcase}
\begin{proof}

The algebra $\mathbb{C}[\mathbb{C}^2] \rtimes G$ can be written as $\frac{T_{\mathbb{C}G} (\mathbb{C}^{2*} \otimes_{\mathbb{C}} \mathbb{C}G)}{<[x,y] \otimes 1>}$ and is Morita equivalent to a path algebra with relations $\mathbb{C}Q/R=\frac{T_S(V)}{R}=\mathcal{D}(\Phi_2,0)$ for some twisted homogeneous superpotential $\Phi_2$, where we use notation as in as in Section \ref{quivers}. In particular the Morita equivalence switches $\mathbb{C}G$ with $S$, and respects the gradings and Koszul resolutions.  Hence considering PBW deformations as in Section \ref{PBW deformations} we see that under the Morita equivalence any PBW deformation of $\mathbb{C}[\mathbb{C}^2] \rtimes G$ would give a PBW deformation of $\mathbb{C}Q/R$, noting that in one case considering PBW deformations relative to $\mathbb{C}G$, and in the other relative to $S$.

Hence it is enough to show that the Morita equivalent twisted superpotential algebra $A:=\mathcal{D}(\Phi_2,0)$ has no nontrivial PBW deformations. 

There can only possibly exist PBW deformations if there exists some non zero $\theta_0,\theta_1$ as in Section \ref{PBW deformations}. But $\theta_1 \in \Hom_{S^e}(R,V)$ and $\theta_0 \in \Hom_{S^e}(R,S)$, so if both these sets are $\{0\}$ there are no nontrivial PBW deformations.

Define the \emph{distance} between two vertices in the quiver to be the minimal length of a path from one to the other. It is shown in the Appendix, Lemma \ref{Class}, that the tail and head of any relation are vertices which are distance greater than one apart. Hence, as $S^e$ module maps preserve heads and tails, the sets $\Hom_{S^e}(R,V)$ and $\Hom_{S^e}(R,S)$ are both $\{0\}$ and there are no nontrivial PBW deformations.
\end{proof}

We look at examples of a small and non small subgroup, using the calculations from \cite{BSW}. We let $\varepsilon_m$ denote a primitive $m^{th}$ root of unity.

\begin{gl2} \label{gl2}
We first consider a small subgroup $\mathbb{D}_{5,2}$ with representation as
\begin{equation*}
\left\langle \left( \begin{array}{c c} \varepsilon_4 & 0 \\ 0 & \varepsilon_4^{-1} \end{array} \right) , \left( \begin{array}{c c} 0 &\varepsilon_4 \\ \varepsilon_4 & 0 \end{array} \right), \left( \begin{array}{c c} 0 &\varepsilon_6 \\ \varepsilon_6 & 0 \end{array} \right) \right \rangle
\end{equation*}

This is given as an example in \cite[Example 5.4]{BSW}.

It has McKay quiver

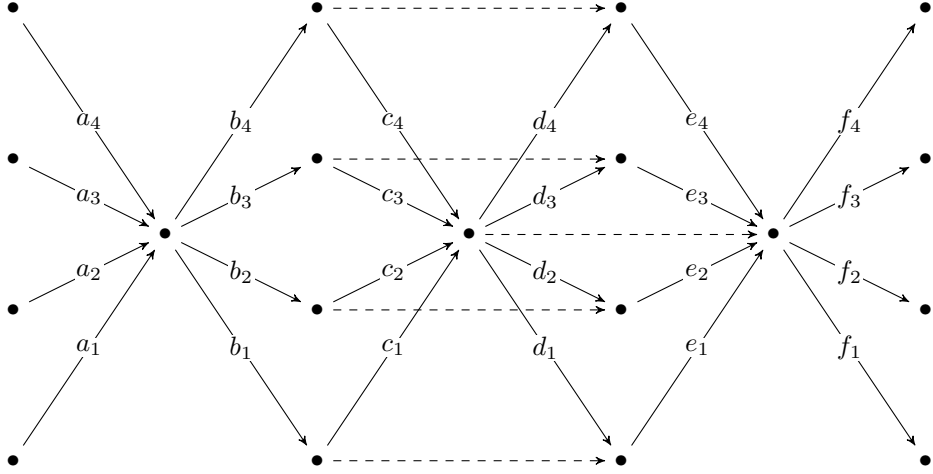
\begin{tikzpicture} [bend angle=0, looseness=1]
\node (A1) at (0,-4)  {$\bullet$};
\node (A2) at (0,-2)  {$\bullet$};
\node (A3) at (0, 0){$\bullet$};
\node (A4) at (0,2) {$\bullet$};

\node (B1) at (2,-1) {$\bullet$};

\node (C1) at (4,-4)  {$\bullet$};
\node (C2) at (4,-2)  {$\bullet$};
\node (C3) at (4, 0){$\bullet$};
\node (C4) at (4,2) {$\bullet$};

\node (D1) at (6,-1) {$\bullet$};

\node (E1) at (8,-4)  {$\bullet$};
\node (E2) at (8,-2)  {$\bullet$};
\node (E3) at (8, 0){$\bullet$};
\node (E4) at (8,2) {$\bullet$};

\node (F1) at (10,-1) {$\bullet$};

\node (G1) at (12,-4)  {$\bullet$};
\node (G2) at (12,-2)  {$\bullet$};
\node (G3) at (12, 0){$\bullet$};
\node (G4) at (12,2) {$\bullet$};

\draw [->,bend left] (A1) to node[gap] {$a_1$}   (B1);
\draw [->,bend left] (A2) to node[gap] {$a_2$} (B1);
\draw [->,bend left] (A3) to node[gap]  {$a_3$} (B1);
\draw [->,bend left] (A4) to node[gap] {$a_4$} (B1);

\draw [->,bend left] (B1) to node[gap]  {$b_1$} (C1);
\draw [->,bend left] (B1) to node[gap] {$b_2$} (C2);
\draw [->,bend left] (B1) to node[gap]  {$b_3$} (C3);
\draw [->,bend left] (B1) to node[gap] {$b_4$} (C4);

\draw [->,bend left] (C1) to node[gap]  {$c_1$} (D1);
\draw [->,bend left] (C2) to node[gap] {$c_2$} (D1);
\draw [->,bend left] (C3) to node[gap]  {$c_3$} (D1);
\draw [->,bend left] (C4) to node[gap] {$c_4$} (D1);

\draw [->,bend left] (D1) to node[gap]  {$d_1$} (E1);
\draw [->,bend left] (D1) to node[gap] {$d_2$} (E2);
\draw [->,bend left] (D1) to node[gap]  {$d_3$} (E3);
\draw [->,bend left] (D1) to node[gap] {$d_4$} (E4);

\draw [->,bend left] (E1) to node[gap]  {$e_1$} (F1);
\draw [->,bend left] (E2) to node[gap] {$e_2$} (F1);
\draw [->,bend left] (E3) to node[gap]  {$e_3$} (F1);
\draw [->,bend left] (E4) to node[gap] {$e_4$} (F1);

\draw [->,bend left] (F1) to node[gap]  {$f_1$} (G1);
\draw [->,bend left] (F1) to node[gap] {$f_2$} (G2);
\draw [->,bend left] (F1) to node[gap]  {$f_3$} (G3);
\draw [->,bend left] (F1) to node[gap] {$f_4$} (G4);

\draw [->, dashed, bend left] (C1) to node {} (E1);
\draw [->, dashed, bend left] (C2) to node {} (E2);
\draw [->, dashed, bend left] (C3) to node {} (E3);
\draw [->, dashed, bend left] (C4) to node {} (E4);

\draw [->, dashed, bend left] (D1) to node {} (F1);

\end{tikzpicture} \\
where the first column of vertices equals the final column of vertices. Tensoring by the determinant representation maps from a vertex to the next vertex in line to the right, as indicated by the dashed arrows, wrapping around from the right side of the diagram to the left.

The relations are 
\begin{equation*}
b_ia_i=0, \quad d_ic_i=0, \quad  f_ie_i=0 \quad \text{for $i=1,2,3,4$}
\end{equation*}
and 
\begin{equation*}
\sum_{j=1}^4 c_j b_j =0, \quad \sum_{j=1}^4 e_j d_j=0, \quad \sum_{j=1}^4 a_j f_j =0 
\end{equation*}

Consider the head and tail of any element of $R$. These vertices are related by the twist, and we see the shortest path between the tail and head is always length 2. Then $\Hom_{S^e}(R, V)$ and $\Hom_{S^e}(R,S)$ are both $\{0\}$, hence there can be no nontrivial PBW deformations.
\end{gl2}

\begin{gl1}\label{gl1}
We now consider $D_8$ as the non-small subgroup with representation
\begin{equation*}
\left\langle g=\left( \begin{array}{c c} \varepsilon_4 & 0 \\ 0 & \varepsilon_4^3 \end{array} \right), \, h=\left( \begin{array}{c c} 0 & 1 \\ 1 & 0 \end{array} \right) \right\rangle
\end{equation*}
and McKay quiver and relations, $Q$ and $R$.
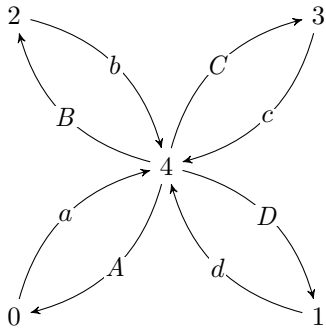
\begin{center}
$
\begin{tikzpicture} [bend angle=30, looseness=1]
\node (C1) at (0,-2)  {$0$};
\node (C2) at (2,0)  {$4$};
\node (C3) at (4, -2){$1$};
\node (C4) at (0,2) {$2$};
\node (C5) at (4,2) {$3$};
\draw [->,bend left] (C1) to node[gap]  {\emph{$a$}} (C2);
\draw [->,bend left] (C2) to node[gap] {\emph{A}} (C1);
\draw [->,bend left] (C3) to node[gap]  {\emph{$d$}} (C2);
\draw [->,bend left] (C2) to node[gap] {\emph{D}} (C3);
\draw [->,bend left] (C4) to node[gap]  {\emph{$b$}} (C2);
\draw [->,bend left] (C2) to node[gap] {\emph{B}} (C4);
\draw [->,bend left] (C5) to node[gap]  {\emph{$c$}} (C2);
\draw [->,bend left] (C2) to node[gap] {\emph{C}} (C5);
\end{tikzpicture}
$
\begin{eqnarray*}
Da=0 \quad Cb=0 \\
Ad=0  \quad Bc=0 \\
aA+dD-bB-cC=0 
\end{eqnarray*}

\end{center}
Then PBW deformations of $\mathbb{C}Q/R$ are classified by $\theta_1,\theta_0:R \rightarrow A$ which are $S^e$-module maps; $\theta_1 \in \Hom_{S^e}(R,V)$ and  $\theta_0 \in \Hom_{S^e}(R,S)$.

As $S^e$-module maps preserve heads and tails we see that $\theta_1$ must be zero, and $\theta_0$ must be zero on all relations but the central one $aA+dD-bB-cC$. At this relation $\theta_0(aA+dD-bB-cC)=\lambda e_{4} $ for some $\lambda \in \mathbb{C}$. In this case $(R\otimes_S V) \cap (V \otimes_S R)=0$, and so any such $\theta_0$ gives us a PBW deformation.  Hence there is a one parameter collection of PBW deformations.
\end{gl1}

\begin{appendix}
\section{McKay quivers for finite small subgroups of $\GL_2(\mathbb{C})$}
We use the classification of McKay quivers for small finite subgroups of $\GL_2(\mathbb{C})$, \cite{AR}, to prove the following:
\begin{Class} \label{Class}
Let $G<\GL_2(\mathbb{C})$ be a small finite subgroup, given by a representation $W \cong \mathbb{C}^2$. Let $W_i$ be an irreducible representation of $G$. Then the shortest path from $W_i$ to $\det W \otimes W_i$ has length $\ge 2$.
\end{Class}

\begin{proof}

We outline this  case by case by examining the McKay quivers, showing there are no length 0 or 1 paths between a vertex in the quiver and the vertex related by tensoring by the determinant.

We list the small finite subgroups of $\GL_2(\mathbb{C})$ up to conjugacy, as in \cite[Section 2]{AR}.

Let $Z_n = \left \langle g=\left( \begin{array}{c c} \varepsilon & 0 \\ 0 & \varepsilon \end{array} \right) \right \rangle$, and $\varepsilon$ be a primitive $n^{th}$ root of unity. Any finite small subgroup of $\GL_2(\mathbb{C})$ is, up to conjugacy, one of the following:
\\

1. $Z$ a cyclic subgroup,  $Z = \left \langle g=\left( \begin{array}{c c} \varepsilon & 0 \\ 0 & \varepsilon^{q} \end{array} \right) \right \rangle \quad \text{ for  $1 \le q<n$.} $
\\

2. $Z_nD= \{ zd \, \, \,| \, \, z \in Z_n, d \in D\}$ for $D$ a finite, non cyclic, subgroup of $\SL_2(\mathbb{C})$.
\\

3. $\bar{H}$. To define $\bar{H}$ let $D < SL_2(\mathbb{C})$ be a binary dihedral group, with $A$ a cyclic subgroup of index 2, and define $H < Z_{2n} \times D$  to be
\begin{equation*}
H=\{ (z,d) \in Z \times D \, | \, d+A=z+Z_m \text{ in } Z_2=D/A=Z_{2n}/Z_{n} \}.
\end{equation*}
Then $\bar{H}$ is the image of $H$ under the map $H \rightarrow GL_2(\mathbb{C})$.
\\

4. $\bar{K}$. To define $\bar{K}$ let $D < SL_2(\mathbb{C})$ be the binary tetrahedral group, with $A$ a normal binary dihedral subgroup of index 3, let $n \ge 3$,  and define $K < Z_{3n} \times D$  to be
\begin{equation*}
K=\{ (z,d) \in Z \times D \, | \,  d+A=z+Z_n \text{ in } Z_3=D/A=Z_{3n}/Z_{n} \}.
\end{equation*}
Then $\bar{K}$ is the image of $K$ under the map $K \rightarrow GL_2(\mathbb{C})$.

We note that if $n=1,2$ then $\bar{K}$ is the binary tetrahedral group with defining representation containing pseudo reflections, so is not small.
\\

The McKay quivers for these groups are described in \cite[Proposition 7]{AR}, and we look at the determinant representation in each case and show tensoring by it relates vertices distance two apart.

We first look at cyclic subgroups. Let $Z$ be as above.  Such a representation is in $\SL_2(\mathbb{C})$ only when $q+1 =n $. We suppose $q+1 \neq n$.

Such a group has $n$ irreducible one dimensional representations, which we label $W_0, \dots W_{n-1}$, where $W_i$ is given by $g \mapsto \varepsilon^i$. The defining  representation is reducible as $\mathbb{C}^2=W_1 \oplus W_q$, and its determinant is the representation $W_{1+q}$. Hence the McKay quiver has $n$ vertices corresponding to the $W_i$ and at vertex $i$ has two arrows to vertices $i+1$ and $i+q$ modulo $n$. The relations on the McKay quiver have head and tail related by tensoring by the determinant, hence any relations with tail $W_i$ have head $W_{i+q+1}$ module $n$. We see that the two vertices $W_i$ and $W_{i+q+1}$ are distance 2 apart; they are not distance 0 as  $i \neq i+1+q$ module $n$, and they are not distance 1 as the only arrows from $i$ are to $i+1$ or $i+q$, and neither of these equals $i+1+q$ modulo $n$.

All the remaining groups are constructed by taking a subgroup of $Z_n \times D$ and then taking the image of this under the map to $\GL_2(\mathbb{C})$. If we calculate the McKay quiver for the subgroup then the image in $\GL_2(\mathbb{C})$ has McKay quiver which is a subquiver. Hence for our purposes it is enough to calculate the McKay quivers for the various subgroups of $Z_n \times D$. 

We first do this for the case $Z_nD$. We note that this is contained in $\SL_2(\mathbb{C})$ for $n=1,2$, hence we assume $n > 2$. In this case we consider the McKay quiver of  $Z_n \times D$. Let the irreducible representations of $D$ be labeled $D_0, \dots, D_{r-1}$ where $D_0$ is trivial, and $D_1$ is the given 2 dimensional representation. Let $R_i$, for $i=0, \dots n-1$, be the $n$ one dimensional irreducible representations of $Z_n$ with $R_i$ given by $g \mapsto \varepsilon^i$. Then $Z_n \times D$ has $nr$ irreducible representations given by $R_i \otimes D_j$ for $0 \le i <n$ and  $0 \le j <r$. Then we consider the McKay quiver for the defining representation $R_1 \otimes D_1$, as this corresponds to the defining representation in $\GL_2(\mathbb{C})$. In particular the McKay quiver has $n$ groups of $r$ representations labeled by representation of $Z_n$, with group $i$ corresponding to the set of representations $\{R_i \otimes D_j \, | \, j=0, \dots r-1 \}$. By definition any arrows in the quiver go from group $i$ to group $i+1$ modulo $n$. As the defining representation is $R_1 \otimes D_1$ the determinant representation is $R_2 \otimes D_0$, and the determinant maps from group $i$ to group $i+2$ modulo $n$. In particular, as $n >2$, any two vertices related by this are not distance zero or one apart.

For the cases $\bar{H}$ and $\bar{K}$ we take the McKay quiver for $Z_n \times D$, make some identifications to account for certain irreducible representations being identified for the subgroups $H,K$.

We first consider $H$. In this case we label the representations of the binary dihedral group as
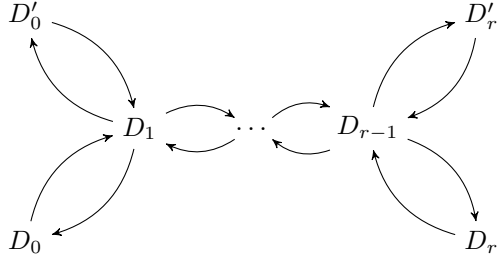
\begin{center}
$
\begin{tikzpicture} [bend angle=30, looseness=1]
\node (C1) at (0,0)  {$D_0$};
\node (C2) at (0,3) {$D_0'$};
\node (C3) at (1.5,1.5)  {$D_1$};
\node (C4) at (3,1.5) { \dots};
\node (C5) at (4.5, 1.5){$D_{r-1}$};
\node (C6) at (6,0)  {$D_r$};
\node (C7) at (6,3) {$D_r'$};

\draw [->,bend left] (C1) to node  {} (C3);
\draw [->,bend left] (C3) to node {}(C1);
\draw [->,bend left] (C2) to node  {} (C3);
\draw [->,bend left] (C3) to node {} (C2);
\draw [->,bend left] (C3) to node  {} (C4);
\draw [->,bend left] (C4) to node {} (C3);
\draw [->,bend left] (C4) to node  {} (C5);
\draw [->,bend left] (C5) to node {} (C4);
\draw [->,bend left] (C5) to node  {} (C6);
\draw [->,bend left] (C5) to node {} (C7);
\draw [->,bend left] (C6) to node  {} (C5);
\draw [->,bend left] (C7) to node {} (C5);
\end{tikzpicture}
$
\end{center}
 where $D_1$ is the representation in $\SL_2(\mathbb{C})$ and $D_0$ is the trivial representation.

Now we label the representations of $Z_{2n}$ as $R_{i}$ for $i=0 \dots 2n-1$ as above, and $Z_{2n} \times D$ has $2n(r+2)$ irreducible representations given by their tensor products. The defining representation in $\GL_2(\mathbb{C})$ is given as $R_1 \otimes D_1$, and hence the determinant representation is $R_2 \otimes D_0$.  Now all representations of $Z_{2n} \times D$ are still irreducible for $H$ but some are identified, \cite[Proposition 7 (d)]{AR}. The representations which are identified are $R_i \otimes D_j$ with $R_{n+i} \otimes  D_j$ for $j=1 \dots r-1$, modulo $2n$, and $R_i \otimes D_j$ with $R_{n+i} \otimes D_j'$ modulo $2n$ for $j=0,r$. The McKay quiver for $H$ is the McKay quiver for $Z_{2n} \times D$  with these identifications made. In this case we group the irreducible representations into $n$ groups labeled by the representation of $Z_{2n}$ modulo $n$, so arrows go from group $i$ to $i+1$ and determinant from group $i$ to $i+2$ module $n$. Hence, for $n>2$, the head and tail of two vertices related by the determinant are not distance zero or one apart. When $n=1,2$ then $H=D$ and then the representation is contained in $\SL_2(\mathbb{C})$.

The final case is to consider $K < Z_{3n} \times D$. Again we label the representations of $Z_{3n}$ by $R_i$  for $i=0 \dots 3n-1$, and we label the representations of $D$ by 

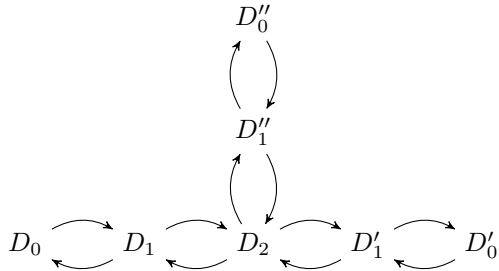
\begin{center}
$
\begin{tikzpicture} [bend angle=30, looseness=1]
\node (C1) at (0,0)  {$D_0$};
\node (C2) at (1.5,0) {$D_1$};
\node (C3) at (3,0)  {$D_2$};
\node (C4) at (4.5,0) { $D_1'$};
\node (C5) at (6, 0){$D_0'$};
\node (C6) at (3,1.5)  {$D_1''$};
\node (C7) at (3,3) {$D_0''$};

\draw [->,bend left] (C1) to node  {} (C2);
\draw [->,bend left] (C2) to node {}(C1);
\draw [->,bend left] (C2) to node  {} (C3);
\draw [->,bend left] (C3) to node {} (C2);
\draw [->,bend left] (C3) to node  {} (C4);
\draw [->,bend left] (C4) to node {} (C3);
\draw [->,bend left] (C4) to node  {} (C5);
\draw [->,bend left] (C5) to node {} (C4);
\draw [->,bend left] (C3) to node  {} (C6);
\draw [->,bend left] (C6) to node {} (C3);
\draw [->,bend left] (C6) to node  {} (C7);
\draw [->,bend left] (C7) to node {} (C6);
\end{tikzpicture}
$
\end{center}
where $D_0$ is the trivial representation, and $D_1$ the defining representation.

The representation defining the group in $\GL_2(\mathbb{C})$ is $R_1 \otimes D_1$, and the determinant representation is $R_2 \otimes D_0$.

Now all the irreducible representations for $Z_{3n} \times D$ remain so for $K$, however some are identified \cite[Proposition 7 (f)]{AR}. This time triples are identified: $R_i \otimes D_j \cong R_{i+n} \otimes D_j' \cong R_{i+2n} \otimes D_j''$ modulo $3n$, for $j=0,1$ and $R_i \otimes D_2 \cong R_{i+n} \otimes D_2 \cong R_{i+2n} \otimes D_2$ modulo $3n$, as representations of $K$.

Once again we note that this splits the quiver into $n$ groups, labeled by the representation of $Z_{3n}$ module $n$, with arrows from group $i$ to $i+1$ and determinant from $i$ to $i+2$ modulo $n$. Hence, as $n>2$, the determinant maps between vertices which are not distance 0 or 1 apart. 

Hence for any small finite subgroup of $\GL_2(\mathbb{C})$ not contained in $\SL_2(\mathbb{C})$ the determinant in the McKay quiver maps between vertices which are distance greater than $1$ apart.

\end{proof}

We also note that, while we do not know a complete proof without the use of the classification, given  $W_i \neq \det W \otimes W_i$ there is an elementary character theoretic proof of Lemma \ref{Class}.

\end{appendix}
\bibliographystyle{plain}
\bibliography{cybib}

\end{document}